\documentclass{amsart}
\usepackage[latin1]{inputenc} 
\usepackage{amssymb,amsmath,amsthm,amscd,epsf,latexsym,verbatim,graphicx,amsfonts}
\input epsf.tex
\usepackage[english]{babel}
\usepackage{enumerate}
\usepackage{enumitem}
\usepackage{mathrsfs}
\usepackage{tikz-cd}
\usepackage{hyperref}
\usepackage[capitalise,noabbrev]{cleveref}
\usepackage{tabularx}
\usepackage{tikz}
\usetikzlibrary{shapes,arrows}
\usepackage{caption}
\usepackage{subcaption}
\usepackage{csquotes}

\newcounter{Theorem}[section]
\numberwithin{Theorem}{section}

\newtheorem{problem}[subsection]{Problem}
\newtheorem{theorem}[Theorem]{Theorem}

\newtheorem{corollary}[Theorem]{Corollary}
\newtheorem{proposition}[Theorem]{Proposition}
\newtheorem{lemma}[Theorem]{Lemma}

\newtheorem{observation}[Theorem]{Observation}

\theoremstyle{definition}
\newtheorem{definition}[Theorem]{Definition}

\newtheorem{remark}[Theorem]{Remark}

\usepackage{comment}
\usepackage[backend=biber,style=numeric,maxbibnames=99]{biblatex}
\renewbibmacro{in:}{}
\DeclareMathOperator{\diam}{\operatorname{diam}}

\DeclareMathOperator{\ecc}{\operatorname{ecc}}
\usepackage{lineno}

\title[A characterization of graphs with at most four boundary vertices]{A characterization of graphs with\\at most four boundary vertices}

\author[N. Chiem]{Nick Chiem}
\address{Department of Mathematics, University of California, Riverside, CA 92521, USA}
\email{nchie005@ucr.edu}

\author[W. Dudarov]{William Dudarov}
\address{Department of Mathematics, University of Washington, Seattle, WA 98125, USA}
\email{wdudarov@gmail.com}

\author[C. Lee]{Chris Lee}
\address{Department of Mathematics, University of Washington, Seattle, WA 98125, USA}
\email{chris.lee.math@gmail.com}

\author[S. Lee]{Sean Lee}
\address{Department of Mathematics, University of Washington, Seattle, WA 98125, USA}
\email{seanlee.beaverton.or@gmail.com}

\author[K. Liu]{Kevin Liu}
\address{Department of Mathematics, University of Washington, Seattle, WA 98125, USA}
\email{kliu15@uw.edu}


\begin{document}


\maketitle

\vspace{-0.3in}
\begin{abstract}
    Steinerberger defined a notion of boundary for a graph and established a corresponding isoperimetric inquality. Hence, ``large" graphs have more boundary vertices. In this paper, we first characterize graphs with three boundary vertices in terms of two infinite families of graphs. We then completely characterize graphs with four boundary vertices in terms of eight families of graphs, five of which are infinite.
    This parallels earlier work by Hasegawa and Saito as well as M\"uller, P\'or, and Sereni on another notion of boundary defined by Chartrand, Erwin, Johns, and Zhang.
\end{abstract}

\section{Introduction}

A priori, the boundary of a graph is not a meaningful concept, as graphs do not have interiors or complements. However, the boundary for subsets in $\mathbb{R}^n$ is induced by a metric, which graphs are also equipped with. Hence, one can ask more generally on any metric space how to find an axiomatic approach to defining subsets that behave in a boundary-like manner. Chartrand, Erwins, Johns, and Zhang proposed one definition for graphs in \cite{chartrand}, and Steinerberger proposed another in \cite{steinerberger}.

\begin{figure}[h!]
    \centering
    \scalebox{0.3}{
    \begin{tikzpicture}
    \node[shape=circle,draw=black, fill=blue] (00) at (0,0) {};
    \node[shape=circle,draw=black, fill=blue] (10) at (2,0) {};
    \node[shape=circle,draw=black, fill=blue] (20) at (4,0) {};
    \node[shape=circle,draw=black, fill=blue] (01) at (0,2) {};
    \node[shape=circle,draw=black, fill=blue] (11) at (2,2) {};
    \node[shape=circle,draw=black, fill=blue] (21) at (4,2) {};
    \node[shape=circle,draw=black, fill=blue] (02) at (0,4) {};
    \node[shape=circle,draw=black, fill=blue] (12) at (2,4) {};
    \node[shape=circle,draw=black, fill=blue] (22) at (4,4) {};
    \node[shape=circle,draw=black, fill=blue] (30) at (6,0) {};
    \node[shape=circle,draw=black, fill=blue] (31) at (6,2) {};
    \node[shape=circle,draw=black, fill=blue] (32) at (6,4) {};
    \node[shape=circle,draw=black, fill=red] (n11) at (-1,-1) {};
    \node[shape=circle,draw=black, fill=red] (44) at (7,7) {};
    \node[shape=circle,draw=black, fill=red] (n14) at (-1,7) {};
    \node[shape=circle,draw=black, fill=red] (n41) at (7,-1) {};
    \node[shape=circle,draw=black, fill=blue] (0515) at (1,3) {};
    \node[shape=circle,draw=black, fill=blue] (2505) at (5,1) {};
    \node[shape=circle,draw=black, fill=blue] (0505) at (1,1) {};
    \node[shape=circle,draw=black, fill=blue] (2515) at (5,3) {};
    \node[shape=circle,draw=black, fill=blue] (1505) at (3,1) {};
    \node[shape=circle,draw=black, fill=blue] (2525) at (5,5) {};
    \node[shape=circle,draw=black, fill=blue] (1525) at (3,5) {};
    \node[shape=circle,draw=black, fill=blue] (0525) at (1,5) {};
    \node[shape=circle,draw=black, fill=blue] (03) at (0,6) {};
    \node[shape=circle,draw=black, fill=blue] (13) at (2,6) {};
    \node[shape=circle,draw=black, fill=blue] (23) at (4,6) {};
    \node[shape=circle,draw=black, fill=blue] (33) at (6,6) {};
    \node[shape=circle,draw=black, fill=blue] (1515) at (3,3) {};
    \path (1515) edge node {} (11);
    \path (1515) edge node {} (21);
    \path (1515) edge node {} (12);
    \path (1515) edge node {} (22);
    \path (1515) edge node {} (1505);
    \path (1515) edge node {} (2515);
    \path (1515) edge node {} (0515);
    \path (1515) edge node {} (1525);
    \path (0505) edge node {} (1505);
    \path (1505) edge node {} (2505);
    \path (0505) edge node {} (0515);
    \path (2505) edge node {} (2515);
    \path (2515) edge node {} (2525);
    \path (0515) edge node {} (0525);
    \path (0525) edge node {} (1525);
    \path (1525) edge node {} (2525);
    \path (00) edge node {} (10);
    \path (10) edge node {} (20);
    \path (00) edge node {} (01);
    \path (01) edge node {} (11);
    \path (11) edge node {} (21);
    \path (02) edge node {} (12);
    \path (12) edge node {} (22);
    \path (01) edge node {} (02);
    \path (21) edge node {} (22);
    \path (20) edge node {} (21);
    \path (10) edge node {} (11);
    \path (11) edge node {} (12);
    \path (22) edge node {} (32);
    \path (21) edge node {} (31);
    \path (20) edge node {} (30);
    \path (30) edge node {} (31);
    \path (31) edge node {} (32);
    \path (20) edge node {} (2505);
    \path (30) edge node {} (2505);
    \path (21) edge node {} (2505);
    \path (31) edge node {} (2505);
    \path (10) edge node {} (1505);
    \path (20) edge node {} (1505);
    \path (11) edge node {} (1505);
    \path (21) edge node {} (1505);
    \path (01) edge node {} (0515);
    \path (11) edge node {} (0515);
    \path (02) edge node {} (0515);
    \path (12) edge node {} (0515);
    \path (00) edge node {} (0505);
    \path (10) edge node {} (0505);
    \path (01) edge node {} (0505);
    \path (11) edge node {} (0505);
    \path (21) edge node {} (2515);
    \path (31) edge node {} (2515);
    \path (22) edge node {} (2515);
    \path (32) edge node {} (2515);
    \path (02) edge node {} (0525);
    \path (03) edge node {} (0525);
    \path (12) edge node {} (0525);
    \path (13) edge node {} (0525);
    \path (12) edge node {} (1525);
    \path (13) edge node {} (1525);
    \path (22) edge node {} (1525);
    \path (23) edge node {} (1525);
    \path (22) edge node {} (2525);
    \path (23) edge node {} (2525);
    \path (32) edge node {} (2525);
    \path (33) edge node {} (2525);
    \path (00) edge node {} (n11);
    \path (30) edge node {} (n41);
    \path (03) edge node {} (n14);
    \path (33) edge node {} (44);
    \path (03) edge node {} (02);
    \path (13) edge node {} (12);
    \path (23) edge node {} (22);
    \path (33) edge node {} (32);
    \path (03) edge node {} (13);
    \path (13) edge node {} (23);
    \path (23) edge node {} (33);
    \node at (0,-2) {};
\end{tikzpicture} 
\qquad \qquad \qquad 
\begin{tikzpicture}
   \node[shape=circle,draw=black, fill=blue] (00) at ({3-2*sqrt(2)},{3-2*sqrt(2)}) {};
    \node[shape=circle,draw=black, fill=blue] (20) at ({3+4*sin(45)/sin(112.5)*cos(11*180/8)},{3+4*sin(45)/sin(112.5)*sin(11*180/8)}) {};
    \node[shape=circle,draw=black, fill=blue] (40) at ({3+4*sin(45)/sin(112.5)*cos(13*180/8)},{3+4*sin(45)/sin(112.5)*sin(13*180/8)}) {};
    \node[shape=circle,draw=black, fill=blue] (02) at ({3+4*sin(45)/sin(112.5)*cos(9*180/8)},{3+4*sin(45)/sin(112.5)*sin(9*180/8)}) {};
    \node[shape=circle,draw=black, fill=blue] (04) at ({3+4*sin(45)/sin(112.5)*cos(7*180/8)},{3+4*sin(45)/sin(112.5)*sin(7*180/8)}) {};
    \node[shape=circle,draw=black, fill=blue] (60) at ({3+2*sqrt(2)},{3-2*sqrt(2)}) {};
    \node[shape=circle,draw=black, fill=blue] (62) at ({3+4*sin(45)/sin(112.5)*cos(15*180/8)},{3+4*sin(45)/sin(112.5)*sin(15*180/8)}) {};
    \node[shape=circle,draw=black, fill=blue] (64) at ({3+4*sin(45)/sin(112.5)*cos(180/8)},{3+4*sin(45)/sin(112.5)*sin(180/8)}) {};
    \node[shape=circle,draw=black, fill=blue] (06) at ({3-2*sqrt(2)},{3+2*sqrt(2)}) {};
    \node[shape=circle,draw=black, fill=blue] (26) at ({3+4*sin(45)/sin(112.5)*cos(5*180/8)},{3+4*sin(45)/sin(112.5)*sin(5*180/8)}) {};
    \node[shape=circle,draw=black, fill=blue] (46) at ({3+4*sin(45)/sin(112.5)*cos(3*180/8)},{3+4*sin(45)/sin(112.5)*sin(3*180/8)}) {};
    \node[shape=circle,draw=black, fill=blue] (66) at ({3+2*sqrt(2)},{3+2*sqrt(2)}) {};
    \node[shape=circle,draw=black, fill=blue] (37) at (3,7) {};
    \node[shape=circle,draw=black, fill=blue] (3n1) at (3,-1) {};
    \node[shape=circle,draw=black, fill=blue] (73) at (7,3) {};
    \node[shape=circle,draw=black, fill=blue] (n13) at (-1,3) {};
    \node[shape=circle,draw=black, fill=red] (38) at (3,8) {};
    \node[shape=circle,draw=black, fill=red] (3n2) at (3,-2) {};
    \node[shape=circle,draw=black, fill=red] (83) at (8,3) {};
    \node[shape=circle,draw=black, fill=red] (n23) at (-2,3) {};
    \node[shape=circle,draw=black, fill=red] (77) at ({3+2.5*sqrt(2)},{3+2.5*sqrt(2)}) {};
    \node[shape=circle,draw=black, fill=red] (7n1) at ({3+2.5*sqrt(2)},{3-2.5*sqrt(2)}) {};
    \node[shape=circle,draw=black, fill=red] (n1n1) at ({3-2.5*sqrt(2)},{3-2.5*sqrt(2)}) {};
    \node[shape=circle,draw=black, fill=red] (n17) at ({3-2.5*sqrt(2)},{3+2.5*sqrt(2)}) {};
    \path (00) edge node {} (20);
    \path (20) edge node {} (40);
    \path (40) edge node {} (60);
    \path (06) edge node {} (26);
    \path (26) edge node {} (46);
    \path (46) edge node {} (66);
    \path (00) edge node {} (02);
    \path (02) edge node {} (04);
    \path (04) edge node {} (06);
    \path (60) edge node {} (62);
    \path (62) edge node {} (64);
    \path (64) edge node {} (66);
    \path (20) edge node {} (02);
    \path (04) edge node {} (26);
    \path (46) edge node {} (64);
    \path (62) edge node {} (40);
    \path (26) edge node {} (37);
    \path (37) edge node {} (38);
    \path (46) edge node {} (37);
    \path (20) edge node {} (3n1);
    \path (3n1) edge node {} (3n2);
    \path (40) edge node {} (3n1);
    \path (02) edge node {} (n13);
    \path (04) edge node {} (n13);
    \path (n13) edge node {} (n23);
    \path (62) edge node {} (73);
    \path (64) edge node {} (73);
    \path (73) edge node {} (83);
    \path (66) edge node {} (77);
    \path (7n1) edge node {} (60);
    \path (n1n1) edge node {} (00);
    \path (n17) edge node {} (06);
    \path (20) edge node {} (46);
    \path (40) edge node {} (26);
    \path (02) edge node {} (64);
    \path (04) edge node {} (62);
\end{tikzpicture}
\qquad \qquad \qquad 
\begin{tikzpicture}
    \node[shape=circle,draw=black, fill=red] (00) at (0,0) {};
    \node[shape=circle,draw=black, fill=blue] (10) at (2,0) {};
    \node[shape=circle,draw=black, fill=red] (20) at (4,0) {};
    \node[shape=circle,draw=black, fill=blue] (30) at (6,0) {};
    \node[shape=circle,draw=black, fill=red] (40) at (8,0) {};
    \node[shape=circle,draw=black, fill=blue] (01) at (0,2) {};
    \node[shape=circle,draw=black, fill=blue] (11) at (2,2) {};
    \node[shape=circle,draw=black, fill=blue] (21) at (4,2) {};
    \node[shape=circle,draw=black, fill=blue] (31) at (6,2) {};
    \node[shape=circle,draw=black, fill=blue] (41) at (8,2) {};
    \node[shape=circle,draw=black, fill=red] (02) at (0,4) {};
    \node[shape=circle,draw=black, fill=blue] (12) at (2,4) {};
    \node[shape=circle,draw=black, fill=blue] (32) at (6,4) {};
    \node[shape=circle,draw=black, fill=red] (42) at (8,4) {};
    \node[shape=circle,draw=black, fill=blue] (03) at (0,6) {};
    \node[shape=circle,draw=black, fill=blue] (13) at (2,6) {};
    \node[shape=circle,draw=black, fill=blue] (23) at (4,6) {};
    \node[shape=circle,draw=black, fill=blue] (33) at (6,6) {};
    \node[shape=circle,draw=black, fill=blue] (43) at (8,6) {};
    \node[shape=circle,draw=black, fill=red] (04) at (0,8) {};
    \node[shape=circle,draw=black, fill=blue] (14) at (2,8) {};
    \node[shape=circle,draw=black, fill=red] (24) at (4,8) {};
    \node[shape=circle,draw=black, fill=blue] (34) at (6,8) {};
    \node[shape=circle,draw=black, fill=red] (44) at (8,8) {};
    \node at (0,-1) {};
    \path (00) edge node {} (10);
    \path (10) edge node {} (20);
    \path (20) edge node {} (30);
    \path (30) edge node {} (40);
    
    \path (01) edge node {} (11);
    \path (11) edge node {} (21);
    \path (21) edge node {} (31);
    \path (31) edge node {} (41);
    
    \path (02) edge node {} (12);
    \path (32) edge node {} (42);
    
    \path (03) edge node {} (13);
    \path (13) edge node {} (23);
    \path (23) edge node {} (33);
    \path (33) edge node {} (43);
    
    \path (04) edge node {} (14);
    \path (14) edge node {} (24);
    \path (24) edge node {} (34);
    \path (34) edge node {} (44);
    
    \path (01) edge node {} (02);
    \path (02) edge node {} (03);
    \path (03) edge node {} (04);
    \path (00) edge node {} (01);
    
    \path (11) edge node {} (12);
    \path (12) edge node {} (13);
    \path (13) edge node {} (14);
    \path (10) edge node {} (11);
    
    \path (21) edge node {} (20);
    \path (23) edge node {} (24);
    
    \path (31) edge node {} (32);
    \path (32) edge node {} (33);
    \path (33) edge node {} (34);
    \path (30) edge node {} (31);
    
    \path (41) edge node {} (42);
    \path (42) edge node {} (43);
    \path (43) edge node {} (44);
    \path (40) edge node {} (41);
\end{tikzpicture} 
}\\
\phantom{-} \\
\scalebox{0.3}{
\begin{tikzpicture}
    \node[shape=circle,draw=black, fill=red] (00) at (0,0) {};
    \node[shape=circle,draw=black, fill=red] (10) at (2,0) {};
    \node[shape=circle,draw=black, fill=red] (20) at (4,0) {};
    \node[shape=circle,draw=black, fill=red] (01) at (0,2) {};
    \node[shape=circle,draw=black, fill=blue] (11) at (2,2) {};
    \node[shape=circle,draw=black, fill=blue] (21) at (4,2) {};
    \node[shape=circle,draw=black, fill=red] (02) at (0,4) {};
    \node[shape=circle,draw=black, fill=blue] (12) at (2,4) {};
    \node[shape=circle,draw=black, fill=blue] (22) at (4,4) {};
    \node[shape=circle,draw=black, fill=red] (30) at (6,0) {};
    \node[shape=circle,draw=black, fill=red] (31) at (6,2) {};
    \node[shape=circle,draw=black, fill=red] (32) at (6,4) {};
    \node[shape=circle,draw=black, fill=red] (n11) at (-1,-1) {};
    \node[shape=circle,draw=black, fill=red] (44) at (7,7) {};
    \node[shape=circle,draw=black, fill=red] (n14) at (-1,7) {};
    \node[shape=circle,draw=black, fill=red] (n41) at (7,-1) {};
    \node[shape=circle,draw=black, fill=blue] (0515) at (1,3) {};
    \node[shape=circle,draw=black, fill=blue] (2505) at (5,1) {};
    \node[shape=circle,draw=black, fill=blue] (0505) at (1,1) {};
    \node[shape=circle,draw=black, fill=blue] (2515) at (5,3) {};
    \node[shape=circle,draw=black, fill=blue] (1505) at (3,1) {};
    \node[shape=circle,draw=black, fill=blue] (2525) at (5,5) {};
    \node[shape=circle,draw=black, fill=blue] (1525) at (3,5) {};
    \node[shape=circle,draw=black, fill=blue] (0525) at (1,5) {};
    \node[shape=circle,draw=black, fill=red] (03) at (0,6) {};
    \node[shape=circle,draw=black, fill=red] (13) at (2,6) {};
    \node[shape=circle,draw=black, fill=red] (23) at (4,6) {};
    \node[shape=circle,draw=black, fill=red] (33) at (6,6) {};
    \node[shape=circle,draw=black, fill=blue] (1515) at (3,3) {};
    \path (1515) edge node {} (11);
    \path (1515) edge node {} (21);
    \path (1515) edge node {} (12);
    \path (1515) edge node {} (22);
    \path (1515) edge node {} (1505);
    \path (1515) edge node {} (2515);
    \path (1515) edge node {} (0515);
    \path (1515) edge node {} (1525);
    \path (0505) edge node {} (1505);
    \path (1505) edge node {} (2505);
    \path (0505) edge node {} (0515);
    \path (2505) edge node {} (2515);
    \path (2515) edge node {} (2525);
    \path (0515) edge node {} (0525);
    \path (0525) edge node {} (1525);
    \path (1525) edge node {} (2525);
    \path (00) edge node {} (10);
    \path (10) edge node {} (20);
    \path (00) edge node {} (01);
    \path (01) edge node {} (11);
    \path (11) edge node {} (21);
    \path (02) edge node {} (12);
    \path (12) edge node {} (22);
    \path (01) edge node {} (02);
    \path (21) edge node {} (22);
    \path (20) edge node {} (21);
    \path (10) edge node {} (11);
    \path (11) edge node {} (12);
    \path (22) edge node {} (32);
    \path (21) edge node {} (31);
    \path (20) edge node {} (30);
    \path (30) edge node {} (31);
    \path (31) edge node {} (32);
    \path (20) edge node {} (2505);
    \path (30) edge node {} (2505);
    \path (21) edge node {} (2505);
    \path (31) edge node {} (2505);
    \path (10) edge node {} (1505);
    \path (20) edge node {} (1505);
    \path (11) edge node {} (1505);
    \path (21) edge node {} (1505);
    \path (01) edge node {} (0515);
    \path (11) edge node {} (0515);
    \path (02) edge node {} (0515);
    \path (12) edge node {} (0515);
    \path (00) edge node {} (0505);
    \path (10) edge node {} (0505);
    \path (01) edge node {} (0505);
    \path (11) edge node {} (0505);
    \path (21) edge node {} (2515);
    \path (31) edge node {} (2515);
    \path (22) edge node {} (2515);
    \path (32) edge node {} (2515);
    \path (02) edge node {} (0525);
    \path (03) edge node {} (0525);
    \path (12) edge node {} (0525);
    \path (13) edge node {} (0525);
    \path (12) edge node {} (1525);
    \path (13) edge node {} (1525);
    \path (22) edge node {} (1525);
    \path (23) edge node {} (1525);
    \path (22) edge node {} (2525);
    \path (23) edge node {} (2525);
    \path (32) edge node {} (2525);
    \path (33) edge node {} (2525);
    \path (00) edge node {} (n11);
    \path (30) edge node {} (n41);
    \path (03) edge node {} (n14);
    \path (33) edge node {} (44);
    \path (03) edge node {} (02);
    \path (13) edge node {} (12);
    \path (23) edge node {} (22);
    \path (33) edge node {} (32);
    \path (03) edge node {} (13);
    \path (13) edge node {} (23);
    \path (23) edge node {} (33);
    \node at (0,-2) {};
\end{tikzpicture}  \qquad  \qquad \qquad 
\begin{tikzpicture}
   \node[shape=circle,draw=black, fill=red] (00) at ({3-2*sqrt(2)},{3-2*sqrt(2)}) {};
    \node[shape=circle,draw=black, fill=blue] (20) at ({3+4*sin(45)/sin(112.5)*cos(11*180/8)},{3+4*sin(45)/sin(112.5)*sin(11*180/8)}) {};
    \node[shape=circle,draw=black, fill=blue] (40) at ({3+4*sin(45)/sin(112.5)*cos(13*180/8)},{3+4*sin(45)/sin(112.5)*sin(13*180/8)}) {};
    \node[shape=circle,draw=black, fill=blue] (02) at ({3+4*sin(45)/sin(112.5)*cos(9*180/8)},{3+4*sin(45)/sin(112.5)*sin(9*180/8)}) {};
    \node[shape=circle,draw=black, fill=blue] (04) at ({3+4*sin(45)/sin(112.5)*cos(7*180/8)},{3+4*sin(45)/sin(112.5)*sin(7*180/8)}) {};
    \node[shape=circle,draw=black, fill=red] (60) at ({3+2*sqrt(2)},{3-2*sqrt(2)}) {};
    \node[shape=circle,draw=black, fill=blue] (62) at ({3+4*sin(45)/sin(112.5)*cos(15*180/8)},{3+4*sin(45)/sin(112.5)*sin(15*180/8)}) {};
    \node[shape=circle,draw=black, fill=blue] (64) at ({3+4*sin(45)/sin(112.5)*cos(180/8)},{3+4*sin(45)/sin(112.5)*sin(180/8)}) {};
    \node[shape=circle,draw=black, fill=red] (06) at ({3-2*sqrt(2)},{3+2*sqrt(2)}) {};
    \node[shape=circle,draw=black, fill=blue] (26) at ({3+4*sin(45)/sin(112.5)*cos(5*180/8)},{3+4*sin(45)/sin(112.5)*sin(5*180/8)}) {};
    \node[shape=circle,draw=black, fill=blue] (46) at ({3+4*sin(45)/sin(112.5)*cos(3*180/8)},{3+4*sin(45)/sin(112.5)*sin(3*180/8)}) {};
    \node[shape=circle,draw=black, fill=red] (66) at ({3+2*sqrt(2)},{3+2*sqrt(2)}) {};
    \node[shape=circle,draw=black, fill=red] (37) at (3,7) {};
    \node[shape=circle,draw=black, fill=red] (3n1) at (3,-1) {};
    \node[shape=circle,draw=black, fill=red] (73) at (7,3) {};
    \node[shape=circle,draw=black, fill=red] (n13) at (-1,3) {};
    \node[shape=circle,draw=black, fill=red] (38) at (3,8) {};
    \node[shape=circle,draw=black, fill=red] (3n2) at (3,-2) {};
    \node[shape=circle,draw=black, fill=red] (83) at (8,3) {};
    \node[shape=circle,draw=black, fill=red] (n23) at (-2,3) {};
    \node[shape=circle,draw=black, fill=red] (77) at ({3+2.5*sqrt(2)},{3+2.5*sqrt(2)}) {};
    \node[shape=circle,draw=black, fill=red] (7n1) at ({3+2.5*sqrt(2)},{3-2.5*sqrt(2)}) {};
    \node[shape=circle,draw=black, fill=red] (n1n1) at ({3-2.5*sqrt(2)},{3-2.5*sqrt(2)}) {};
    \node[shape=circle,draw=black, fill=red] (n17) at ({3-2.5*sqrt(2)},{3+2.5*sqrt(2)}) {};
    \path (00) edge node {} (20);
    \path (20) edge node {} (40);
    \path (40) edge node {} (60);
    \path (06) edge node {} (26);
    \path (26) edge node {} (46);
    \path (46) edge node {} (66);
    \path (00) edge node {} (02);
    \path (02) edge node {} (04);
    \path (04) edge node {} (06);
    \path (60) edge node {} (62);
    \path (62) edge node {} (64);
    \path (64) edge node {} (66);
    \path (20) edge node {} (02);
    \path (04) edge node {} (26);
    \path (46) edge node {} (64);
    \path (62) edge node {} (40);
    \path (26) edge node {} (37);
    \path (37) edge node {} (38);
    \path (46) edge node {} (37);
    \path (20) edge node {} (3n1);
    \path (3n1) edge node {} (3n2);
    \path (40) edge node {} (3n1);
    \path (02) edge node {} (n13);
    \path (04) edge node {} (n13);
    \path (n13) edge node {} (n23);
    \path (62) edge node {} (73);
    \path (64) edge node {} (73);
    \path (73) edge node {} (83);
    \path (66) edge node {} (77);
    \path (7n1) edge node {} (60);
    \path (n1n1) edge node {} (00);
    \path (n17) edge node {} (06);
    \path (20) edge node {} (46);
    \path (40) edge node {} (26);
    \path (02) edge node {} (64);
    \path (04) edge node {} (62);
\end{tikzpicture}\qquad \qquad \qquad 
\begin{tikzpicture}
    \node[shape=circle,draw=black, fill=red] (00) at (0,0) {};
    \node[shape=circle,draw=black, fill=red] (10) at (2,0) {};
    \node[shape=circle,draw=black, fill=red] (20) at (4,0) {};
    \node[shape=circle,draw=black, fill=red] (30) at (6,0) {};
    \node[shape=circle,draw=black, fill=red] (40) at (8,0) {};
    \node[shape=circle,draw=black, fill=red] (01) at (0,2) {};
    \node[shape=circle,draw=black, fill=blue] (11) at (2,2) {};
    \node[shape=circle,draw=black, fill=red] (21) at (4,2) {};
    \node[shape=circle,draw=black, fill=blue] (31) at (6,2) {};
    \node[shape=circle,draw=black, fill=red] (41) at (8,2) {};
    \node[shape=circle,draw=black, fill=red] (02) at (0,4) {};
    \node[shape=circle,draw=black, fill=red] (12) at (2,4) {};
    \node[shape=circle,draw=black, fill=red] (32) at (6,4) {};
    \node[shape=circle,draw=black, fill=red] (42) at (8,4) {};
    \node[shape=circle,draw=black, fill=red] (03) at (0,6) {};
    \node[shape=circle,draw=black, fill=blue] (13) at (2,6) {};
    \node[shape=circle,draw=black, fill=red] (23) at (4,6) {};
    \node[shape=circle,draw=black, fill=blue] (33) at (6,6) {};
    \node[shape=circle,draw=black, fill=red] (43) at (8,6) {};
    \node[shape=circle,draw=black, fill=red] (04) at (0,8) {};
    \node[shape=circle,draw=black, fill=red] (14) at (2,8) {};
    \node[shape=circle,draw=black, fill=red] (24) at (4,8) {};
    \node[shape=circle,draw=black, fill=red] (34) at (6,8) {};
    \node[shape=circle,draw=black, fill=red] (44) at (8,8) {};
    \node at (0,-1) {};
    \path (00) edge node {} (10);
    \path (10) edge node {} (20);
    \path (20) edge node {} (30);
    \path (30) edge node {} (40);
    
    \path (01) edge node {} (11);
    \path (11) edge node {} (21);
    \path (21) edge node {} (31);
    \path (31) edge node {} (41);
    
    \path (02) edge node {} (12);
    \path (32) edge node {} (42);
    
    \path (03) edge node {} (13);
    \path (13) edge node {} (23);
    \path (23) edge node {} (33);
    \path (33) edge node {} (43);
    
    \path (04) edge node {} (14);
    \path (14) edge node {} (24);
    \path (24) edge node {} (34);
    \path (34) edge node {} (44);
    
    \path (01) edge node {} (02);
    \path (02) edge node {} (03);
    \path (03) edge node {} (04);
    \path (00) edge node {} (01);
    
    \path (11) edge node {} (12);
    \path (12) edge node {} (13);
    \path (13) edge node {} (14);
    \path (10) edge node {} (11);
    
    \path (21) edge node {} (20);
    \path (23) edge node {} (24);
    
    \path (31) edge node {} (32);
    \path (32) edge node {} (33);
    \path (33) edge node {} (34);
    \path (30) edge node {} (31);
    
    \path (41) edge node {} (42);
    \path (42) edge node {} (43);
    \path (43) edge node {} (44);
    \path (40) edge node {} (41);
\end{tikzpicture} 
}
    \caption{The CEJZ (top) and the Steinerberger (bottom) boundaries of three graphs, with boundary vertices shown in red.}
    \label{boundaryexamples}
\end{figure}
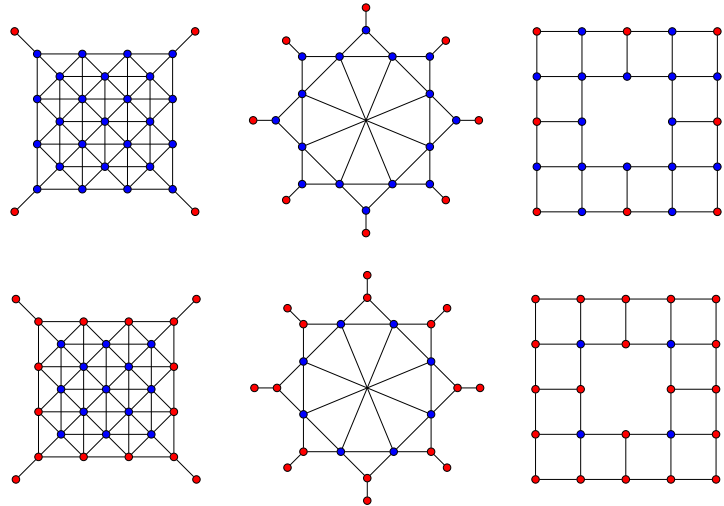

In the Chartrand, Erwin, Johns, and Zhang definition, a vertex $v$ is in the boundary of $G$ if there is a vertex $u$ such that no neighbor of $v$ is farther away from $u$ than $v$. 
However, when considering subsets in $\mathbb{R}^n$ with the Euclidean metric, many boundary points do not satisfy a condition analogous to this. Motivated by this, Steinerberger proposed an alternative notion, where $v$ is a boundary vertex if there is a vertex $u$ such that on average, the neighbors of $v$ are closer to $u$ than $v$ is. 
Examples of both boundaries are shown in \cref{boundaryexamples}.

Formally, the \emph{Chartrand-Erwin-Johns-Zhang (CEJZ) boundary} of a connected graph $G=(V,E)$ is
\begin{equation}
    (\partial G)'=\left\{v \in V\; \bigg\vert \;\exists\;u\in V \text{ such that } d(w,u) \leq d(v,u) \text{ for all } (v,w)\in E \right\}.
\end{equation}
The \emph{Steinerberger boundary} of a connected graph $G = (V, E)$ is defined as
\begin{equation}\label{partialG}
    \partial G = \bigg\{v \in V\; \bigg\vert \;\exists\;u\in V \text{ such that }\frac{1}{\text{deg}(v)}\sum_{(v, w) \in E} d(w,u) < d(v,u)\bigg\},
\end{equation}
where we use the convention $\partial G=V$ when $|V|=1$. We call any $u\in V$ establishing that $v\in \partial G$ or $v\in (\partial G)'$ a \emph{witness} for $v$. 
It is an immediate consequence of the definitions that the condition for $\partial G$ is a relaxation of the condition for $(\partial G)'$.

\begin{proposition}[Steinerberger \cite{steinerberger}, Proposition 1]\label{containment}
For any connected graph $G$, $(\partial G)'\subseteq \partial G$. 
\end{proposition}

The CEJZ boundary has been studied in a variety of contexts, and we refer the reader to \cite{allgeier,  caceres,  caceres3, caceres2, chartrand2, hernando, kang, rodriguez, zejnilovic} for details. When discussing notions of boundary, one natural question is what sets minimize or maximize the size of the boundary. Our work will parallel prior work on characterizing graphs with small CEJZ boundary. The cases of $|(\partial G)'|=2$ or $|(\partial G)'|=3$ were classified by Hasegawa and Saito \cite{hasegawa}, and graphs with $|(\partial G)'|=4$ were classified by M\"{u}ller, P\'{o}r, and Sereni \cite{mullerfour}. The main result of our paper is the following theorem classifying graphs with Steinerberger boundary size at most 4.

\begin{theorem}\label{maintheorem}
Let $G=(V,E)$ be a connected graph. 
\begin{enumerate}[label=(\alph*)]
    \item $|\partial G|=2$ if and only if $G$ is a path graph with at least two vertices.
    \item $|\partial G|=3$ if and only if $G$ is a tree with three leaves or a tripod.
    \item $|\partial G|=4$ if and only if $G$ is a tree with four leaves or one of the graphs in \cref{partialG4} with paths of arbitrary lengths attached to boundary vertices $v$ with boundary stability number $\max_{u\in V} \sum_{w\in N(v)}[d(v,u)-d(w,u)]=1.$
\end{enumerate}
\end{theorem}

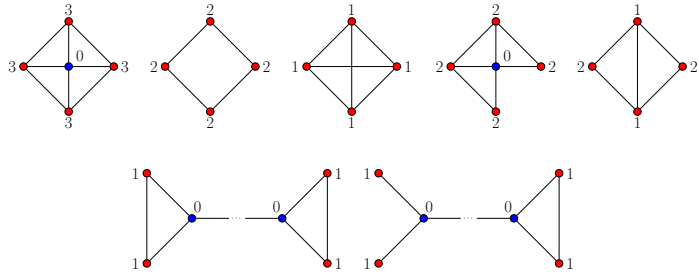
\begin{figure}[h!]
    \centering
    \scalebox{0.3}{\begin{tikzpicture}
    \node[shape=circle,draw=black, fill=red] (A) at (0,2) {};
    \node[shape=circle,draw=black, fill=red] (B) at (-2,0) {};
    \node[shape=circle,draw=black, fill=blue] (M) at (0,0) {};
    \node[shape=circle,draw=black, fill=red] (C) at (0,-2) {};
    \node[shape=circle,draw=black, fill=red] (D) at (2,0) {};
    \node at (0,2.5) {\Huge{$3$}};
    \node at (2.5,0) {\Huge{$3$}};
    \node at (0,-2.5) {\Huge{$3$}};
    \node at (-2.5,0) {\Huge{$3$}};
    \node at (0.5,0.5) {\Huge{$0$}};
    \path (A) edge node {} (M);
    \path (M) edge node {} (D);
    \path (A) edge node {} (D);
    \path (A) edge node {} (B);
    \path (B) edge node {} (C);
    \path (C) edge node {} (D);
    \path (B) edge node {} (M);
    \path (C) edge node {} (M);
\end{tikzpicture}\qquad 
\begin{tikzpicture}
    \node[shape=circle,draw=black, fill=red] (A) at (0,2) {};
    \node[shape=circle,draw=black, fill=red] (B) at (-2,0) {};
    \node[shape=circle,draw=black, fill=red] (C) at (0,-2) {};
    \node[shape=circle,draw=black, fill=red] (D) at (2,0) {};
\node at (0,2.5) {\Huge{$2$}};
    \node at (2.5,0) {\Huge{$2$}};
    \node at (0,-2.5) {\Huge{$2$}};
    \node at (-2.5,0) {\Huge{$2$}};
    \path (A) edge node {} (B);
    \path (A) edge node {} (D);
    \path (B) edge node {} (C);
    \path (C) edge node {} (D);
\end{tikzpicture}\qquad 
\begin{tikzpicture}
    \node[shape=circle,draw=black, fill=red] (A) at (0,2) {};
    \node[shape=circle,draw=black, fill=red] (B) at (-2,0) {};
    \node[shape=circle,draw=black, fill=red] (C) at (0,-2) {};
    \node[shape=circle,draw=black, fill=red] (D) at (2,0) {};
\node at (0,2.5) {\Huge{$1$}};
    \node at (2.5,0) {\Huge{$1$}};
    \node at (0,-2.5) {\Huge{$1$}};
    \node at (-2.5,0) {\Huge{$1$}};
    \path (A) edge node {} (B);
    \path (A) edge node {} (C);
    \path (A) edge node {} (D);
    \path (B) edge node {} (C);
    \path (B) edge node {} (D);
    \path (C) edge node {} (D);
\end{tikzpicture} \qquad 
\begin{tikzpicture}
    \node[shape=circle,draw=black, fill=red] (A) at (0,2) {};
    \node[shape=circle,draw=black, fill=red] (B) at (-2,0) {};
    \node[shape=circle,draw=black, fill=blue] (M) at (0,0) {};
    \node[shape=circle,draw=black, fill=red] (C) at (0,-2) {};
    \node[shape=circle,draw=black, fill=red] (D) at (2,0) {};
\node at (0,2.5) {\Huge{$2$}};
    \node at (2.5,0) {\Huge{$2$}};
    \node at (0,-2.5) {\Huge{$2$}};
    \node at (-2.5,0) {\Huge{$2$}};
    \node at (0.5,0.5) {\Huge{$0$}};
    \path (A) edge node {} (M);
    \path (M) edge node {} (D);
    \path (A) edge node {} (D);
    \path (A) edge node {} (B);
    \path (B) edge node {} (C);
    \path (B) edge node {} (M);
    \path (C) edge node {} (M);
\end{tikzpicture}\qquad 
\begin{tikzpicture}
    \node[shape=circle,draw=black, fill=red] (A) at (0,2) {};
    \node[shape=circle,draw=black, fill=red] (B) at (-2,0) {};
    \node[shape=circle,draw=black, fill=red] (C) at (0,-2) {};
    \node[shape=circle,draw=black, fill=red] (D) at (2,0) {};
\node at (0,2.5) {\Huge{$1$}};
    \node at (2.5,0) {\Huge{$2$}};
    \node at (0,-2.5) {\Huge{$1$}};
    \node at (-2.5,0) {\Huge{$2$}};
    \path (A) edge node {} (B);
    \path (A) edge node {} (C);
    \path (A) edge node {} (D);
    \path (B) edge node {} (C);
    \path (C) edge node {} (D);
\end{tikzpicture}}\\
\phantom{-}\\
\scalebox{0.3}{\begin{tikzpicture}
    \node[shape=circle,draw=black, fill=red] (A) at (-4,2) {};
    \node[shape=circle,draw=black, fill=red] (D) at (-4,-2) {};
    \node[shape=circle,draw=black, fill=blue] (M1) at (-2,0) {};
    \node (P) at (0,0) {$\cdots$};
    \node[shape=circle,draw=black, fill=blue] (M2) at (2,0) {};
    \node[shape=circle,draw=black, fill=red] (C) at (4,-2) {};
    \node[shape=circle,draw=black, fill=red] (B) at (4,2) {};
    \node at (-4.5,2) {\Huge{$1$}};
    \node at (4.5,2) {\Huge{$1$}};
    \node at (-4.5,-2) {\Huge{$1$}};
    \node at (4.5,-2) {\Huge{$1$}};
    \node at (1.75,0.5) {\Huge{$0$}};
    \node at (-1.75,0.5) {\Huge{$0$}};
    \path (A) edge node {} (M1);
    \path (M1) edge node {} (D);
    \path (A) edge node {} (D);
    \path (B) edge node {} (C);
    \path (B) edge node {} (M2);
    \path (C) edge node {} (M2);
    \path (M1) edge node {} (P);
    \path (P) edge node {} (M2);
\end{tikzpicture}\qquad
\begin{tikzpicture}
    \node[shape=circle,draw=black, fill=red] (A) at (-4,2) {};
    \node[shape=circle,draw=black, fill=red] (D) at (-4,-2) {};
    \node[shape=circle,draw=black, fill=blue] (M1) at (-2,0) {};
    \node (P) at (0,0) {$\cdots$};
    \node[shape=circle,draw=black, fill=blue] (M2) at (2,0) {};
    \node[shape=circle,draw=black, fill=red] (C) at (4,-2) {};
    \node[shape=circle,draw=black, fill=red] (B) at (4,2) {};
    \node at (-4.5,2) {\Huge{$1$}};
    \node at (4.5,2) {\Huge{$1$}};
    \node at (-4.5,-2) {\Huge{$1$}};
    \node at (4.5,-2) {\Huge{$1$}};
    \node at (1.75,0.5) {\Huge{$0$}};
    \node at (-1.75,0.5) {\Huge{$0$}};
    \path (A) edge node {} (M1);
    \path (M1) edge node {} (D);
    \path (B) edge node {} (C);
    \path (B) edge node {} (M2);
    \path (C) edge node {} (M2);
    \path (M1) edge node {} (P);
    \path (P) edge node {} (M2);
\end{tikzpicture}}
\caption{Non-tree graphs with $|\partial G|=4$, where $\partial G$ is shown in red. The boundary stability number is given beside each vertex.}
\label{partialG4}
\end{figure}

Note that a \emph{tripod} is a graph formed by starting with the complete graph on three vertices and attaching a path of arbitrary length (possibly 0) to each vertex. Additionally, the boundary stability number is a new parameter that we introduce to establish this theorem. For any $v\in \partial G$, this parameter is an integer that measures how stable the condition in \cref{partialG} is under certain operations. It allows us to more accurately describe what happens when we build larger graphs from smaller ones, which is one of the techniques used in the classification of graphs with small CEJZ boundary.

In Euclidean space, regions with larger volume typically have larger boundaries. One formal statement of this is the isoperimetric 
inequality, which gives a lower bound on the surface area or perimeter of a region in terms of its volume. Steinerberger established a corresponding isoperimetric inequality for $\partial G$ in \cite[Theorem 1]{steinerberger} given by
\[|\partial G|\geq \frac{1}{2\Delta}\cdot \frac{|V|}{\diam(G)},\]
where $\Delta$ is the maximum degree of $G$. Hence, graphs with a large number of vertices have more boundary vertices unless these graphs contain large paths. We see that the results of Theorem \ref{maintheorem} are consistent with the above isoperimetric inequality for $\partial G$. 

Note that the CEJZ boundary does not exhibit this general behavior. From generalizations of the first graph in \cref{boundaryexamples} and other grid-like graphs, $(\partial G)'$ can remain very small for very large graphs that are not formed from paths. A lower bound on $|(\partial G)'|$ in terms of the maximum degree was previously studied in \cite{mullerbound}, where the authors established a bound that is logarithmic with respect to $\Delta$ and showed that this bound is sharp up to constants. However, this bound is entirely independent of $|V|$.

We start in \cref{preliminaries} by summarizing previous results for the CEJZ and Steinerberger boundaries. In particular, we will summarize the characterizations of graphs with small CEJZ boundary. In \cref{stability}, we will introduce the boundary stability number of a vertex and establish several lemmas. We then apply these results in \cref{boundarysize4} to prove \cref{maintheorem}. In \cref{largeboundary}, we apply some of our results to describe some graphs with large Steinerberger boundary. We then conclude with open questions in \cref{futurework}.

\section{Preliminaries}\label{preliminaries}

In this section, we outline necessary definitions and notation, and we then summarize previous results for both the CEJZ and Steinerberger boundaries. Throughout this paper, all graphs are assumed to be simple and undirected. We assume basic familiarity with these graphs and refer the reader to \cite{bona} or \cite{bondy} for this information. We summarize the definitions and notation relevant for this paper. 

Graphs will be denoted $G=(V,E)$. An edge in $E$ will be denoted as a pair, such as $(v,w)$. Note that graphs in this paper are undirected, so this is equivalent to $(w,v)$. We use $N_G(v)$ to denote the set of neighbors of $v$ in the graph $G$ and $\deg_G(v)$ to denote the degree of $v$ in $G$. 

For any graph $G=(V,E)$, a \emph{walk} is a sequence of vertices $W=v_0v_1\ldots v_\ell$ such that $(v_i,v_{i+1})\in E$ for $i=0,1,\ldots,\ell-1$. We call $v_0$ the \emph{initial vertex} and $v_\ell$ the \emph{terminal vertex} of $W$. We also call $\ell$ the \emph{length} of the walk $W$. 
A walk is a \emph{path} if all vertices are distinct, and a \emph{shortest} path from $v$ to $w$ is a path of minimum length with $v$ as its initial vertex and $w$ as its terminal vertex. A walk is \emph{closed} if $v_0=v_{\ell}$, and a closed walk is called a \emph{cycle} if no vertices are repeated except for $v_0=v_{\ell}$. 

A graph $G=(V,E)$ is \emph{connected} if for every $v,w\in V$, there exists a path with $v$ as its initial vertex and $w$ as its terminal vertex. The \emph{connected components} of $G$ are the maximal connected subgraphs of $G$. For a connected graph $G=(V,E)$, a vertex $v\in V$ is a \emph{cut vertex} if $G-v$, the graph obtained by deleting $v$ and all edges incident to it, is not connected.

For a connected graph $G=(V,E)$ and any $v,w\in V$, the \emph{distance} from $v$ to $w$, denoted $d_G(v,w)$, is the length of a shortest path with $v$ as its initial vertex and $w$ as its terminal vertex. It is straightforward to verify that $d_G$ satisfies the properties of a metric. For any $v\in V$, the \emph{eccentricity} of $v$ is
$\ecc_G(v)=\max_{w\in V} \{d_G(v,w)\}$. Note that a vertex has eccentricity $1$ if and only if it is a universal vertex, i.e., a vertex adjacent to all other vertices. The \emph{diameter} of a graph is $\diam(G)=\max_{v\in V} \{\ecc_G(v)\}$.
If $v\in V$ satisfies $\ecc_G(v)=\diam(G)$, then we call $v$ a \emph{peripheral} vertex. When the context is clear, we omit the subscript $G$ in $N_G,\deg_G,d_G$, and $\ecc_G$. An example of these definitions is given in \cref{graphexample}.

\begin{figure}[h!]
\centering
    \scalebox{0.4}{\begin{tikzpicture}
    \node[shape=circle,draw=black, fill=blue] (A) at (0,2) {};
    \node[shape=circle,draw=black, fill=blue] (B) at (-2,0) {};
    \node[shape=circle,draw=black, fill=blue] (C) at (0,-2) {};
    \node[shape=circle,draw=black, fill=blue] (D) at (2,0) {};
    \node[shape=circle,draw=black, fill=blue] (2) at (4,0) {};
    \node[shape=circle,draw=black, fill=blue] (3) at (6,0) {};
    \node[shape=circle,draw=black, fill=blue] (4) at (8,0) {};
    \node[shape=circle,draw=black, fill=blue] (5) at (10,0) {};
    \node[shape=circle,draw=black, fill=blue] (6) at (8,2) {};
    \node[shape=circle,draw=black, fill=blue] (7) at (8,-2) {};
    \path (A) edge node {} (B);
    \path (A) edge node {} (C);
    \path (A) edge node {} (D);
    \path (B) edge node {} (C);
    \path (B) edge node {} (D);
    \path (C) edge node {} (D);
    \path (D) edge node {} (2);
    \path (2) edge node {} (3);
    \path (3) edge node {} (4);
    \path (4) edge node {} (5);
    \path (4) edge node {} (6);
    \path (4) edge node {} (7);
    \node at (-2.75,0) {\Huge{$v_1$}};
    \node at (0,2.75) {\Huge{$v_2$}};
    \node at (0,-2.75) {\Huge{$v_3$}};
    \node at (2,-0.75) {\Huge{$v_4$}};
    \node at (4,-0.75) {\Huge{$v_5$}};
    \node at (6,-0.75) {\Huge{$v_6$}};
    \node at (8,2.75) {\Huge{$v_7$}};
    \node at (8.75,-0.75) {\Huge{$v_8$}};
    \node at (8,-2.75) {\Huge{$v_9$}};
    \node at (11,0) {\Huge{$v_{10}$}};
\end{tikzpicture}}
\caption{A graph $G$ with $\diam(G)=5$, where $v_1,v_2,v_3,v_7,v_9$, and $v_{10}$ are all peripheral vertices. The vertices $v_4,v_5,v_6$, and $v_8$ are cut vertices, as removing any of them disconnects $G$.}
\label{graphexample}
\end{figure}
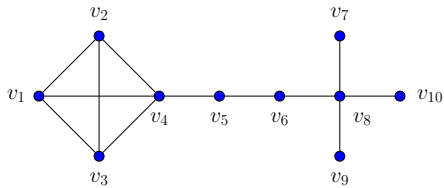

We now summarize previous results. We start with the following  observation, which is used throughout the study of both the CEJZ and Steinerberger boundaries.

\begin{observation}\label{peripheral}
For any connected graph $G=(V,E)$, peripheral vertices are always in $(\partial G)'$. Combined with \cref{containment}, this implies that peripheral vertices are always in $\partial G$. 
\end{observation}

For $\partial G$, the following results were established by Steinerberger in \cite{steinerberger}. The first result immediately implies part (a) of \cref{maintheorem}. Hence, it remains to show parts (b) and (c).

\begin{proposition}[Steinerberger \cite{steinerberger}, Proposition 3]\label{boundary2}
For any connected graph with at least
two vertices, we have $|\partial G|\geq  2$. If $|\partial G|=2$, then $G$ is a path.
\end{proposition}

\begin{proposition}[Steinerberger \cite{steinerberger}, Proposition 2]\label{treeandleaves}
If $G$ is a tree, then $\partial G$ are the vertices of
degree $1$. For any connected graph, vertices of degree $1$ are in $\partial G$. 
\end{proposition}

To establish the remaining parts of our main theorem, we will use the characterization of graphs with small CEJZ boundary, which we summarize here. We start with the results of Hasegawa and Saito characterizing graphs with CEJZ boundary size $2$ or $3$.

\begin{theorem}[Hasegawa and Saito \cite{hasegawa}, Theorem 7]\label{CEJZ2}
Let $G$  be a connected graph. If $|(\partial G)'|=2$, then $G$ is a path.
\end{theorem}

\begin{theorem}[Hasegawa and Saito \cite{hasegawa}, Theorem 9]\label{CEJZ3}
A connected graph $G$ has $|(\partial G)'|=3$ if and only if $G$ is either a tree with three leaves or a
tripod.
\end{theorem}

The characterization of graphs with CEJZ boundary size $4$ is significantly more complex. We start by defining several families of graphs and axis slice convex sets in $\mathbb{R}^2$.

\begin{definition}[M\"{u}ller, P\'{o}r, and Sereni, \cite{mullerfour}, Definition 2]
Let $a$ and $c$ be two positive integers. Define the vertex sets  
\[V_{a\times c}^0=\{(x,y)\in \mathbb{N}^2\, | \, 0\leq x\leq a \text{ and } 0\leq y\leq c\}\]
\[ V_{a\times c}^1=\left\{\left(x+\frac{1}{2},y+\frac{1}{2}\right)\, \bigg| \, (x,y)\in \mathbb{N}^2,0\leq x< a, \text{ and } 0\leq y< c\right\}.\]
\begin{enumerate}[label=(\alph*)]
    \item The grid graph $G_{a\times c}$ has vertex set $V_{a\times c}^0$ and edges between any vertices of Euclidean distance 1. Note that $|V_{a\times c}^0|=(a+1)(c+1)$.
    \item The graph $N_{a\times c}$ has vertex set $V_{a\times c}=V_{a\times c}^0\cup V_{a\times c}^1$ and edges between vertices of Euclidean distance at most $1$.
    \item For $a>2$, the graph $X_{a\times c}$ is the subgraph of $N_{(a-1)\times c}$ induced by \[V_{(a-1)\times c}\setminus \{(x,y)\in \mathbb{N}^2\, | \, 0<x<a-1 \text{ and } y\in \{0,c\}\}.\]
    If $a=2$, then $X_{2\times c}$ is the subgraph of $N_{1\times c}$ obtained by removing the edge between the vertices $(0,0)$ and $(1,0)$, and the edge between the vertices $(0,c)$ and $(1,c)$. If $a\geq 1$ and $c>1$, define $X_{a\times c}=X_{c\times a}$. Finally, define $X_{1\times 1}=K_4$.
    \item The graph $T_{a\times c}$ is the subgraph of $N_{a\times (c+1)}$ induced by 
    \[V_{a\times (c+1)}\setminus (\{(0,y)\, |\, y\in \mathbb{N}\}\cup \{(x,y)\, |\, x<a \text{ and } y\in \{0,c+1\}\}).\]
    \item Let $G_{a\times c}^1$ and $G_{a\times c}^2$ be two grid graphs with vertex sets respectively labeled by $v_{x,y}$ and $w_{x,y}$ for $0\leq x\leq a,0\leq y\leq c$. The graph $D_{a\times c}$ is obtained by identifying $v_{x,y}$ with $w_{x,y}$ whenever $x\in \{0,a\}$ or $y\in \{0,c\}$, adding an edge between $v_{x,y}$ and $w_{x,y}$ for $0\leq x\leq a,0\leq y\leq c$, and adding an edge between $v_{x+1,y}$ and $w_{x,y+1}$ whenever $0\leq x<a$ and $0\leq y<c$. 
    \item The graph $L_{a\times c}$ is obtained from $D_{a\times c}$ by removing the vertices $w_{x,y}$ for $0< x<a$ and $0< y<c$.
\end{enumerate}
\end{definition}

\begin{definition}[M\"{u}ller, P\'{o}r, and Sereni \cite{mullerfour}, Definition 3]
A set $W\subseteq \mathbb{R}^2$ is {axis slice convex} if 
\begin{itemize}
    \item whenever $(x_1,y),(x_2,y)\in W$ and $x_1<x_2$, then $(x,y)\in W$ for all $x$ in $\{x_1,x_1+1,\ldots,x_2\}$, and
    \item whenever $(x,y_1),(x,y_2)\in W$ and $y_1<y_2$, then $(x,y)\in W$ for all $y$ in $\{y_1,y_1+1,\ldots,y_2\}$. 
\end{itemize}
\end{definition}

An example of these definitions is shown in \cref{graphexamples}. We now state the main theorem of M\"{u}ller, P\'{o}r, and Sereni. 

\begin{figure}[h!]
    \centering
    \scalebox{0.4}{
    \begin{tikzpicture}
    \node[shape=circle,draw=black, fill=red] (00) at (0,0) {};
    \node[shape=circle,draw=black, fill=blue] (20) at (2,0) {};
    \node[shape=circle,draw=black, fill=blue] (02) at (0,2) {};
    \node[shape=circle,draw=black, fill=blue] (22) at (2,2) {};
    \node[shape=circle,draw=black, fill=blue] (40) at (4,0) {};
    \node[shape=circle,draw=black, fill=blue] (42) at (4,2) {};
    \node[shape=circle,draw=black, fill=red] (60) at (6,0) {};
    \node[shape=circle,draw=black, fill=blue] (62) at (6,2) {};
    \node[shape=circle,draw=black, fill=red] (04) at (0,4) {};
    \node[shape=circle,draw=black, fill=blue] (24) at (2,4) {};
    \node[shape=circle,draw=black, fill=blue] (44) at (4,4) {};
    \node[shape=circle,draw=black, fill=red] (64) at (6,4) {};
    \node[shape=circle,draw=black, fill=blue] (11) at (1,1) {};
    \node[shape=circle,draw=black, fill=blue] (31) at (3,1) {};
    \node[shape=circle,draw=black, fill=blue] (51) at (5,1) {};
    \node[shape=circle,draw=black, fill=blue] (13) at (1,3) {};
    \node[shape=circle,draw=black, fill=blue] (33) at (3,3) {};
    \node[shape=circle,draw=black, fill=blue] (53) at (5,3) {};
    \path (00) edge node {} (20);
    \path (20) edge node {} (40);
    \path (40) edge node {} (60);
    \path (02) edge node {} (22);
    \path (22) edge node {} (42);
    \path (42) edge node {} (62);
    \path (04) edge node {} (24);
    \path (24) edge node {} (44);
    \path (44) edge node {} (64);
    \path (00) edge node {} (02);
    \path (02) edge node {} (04);
    \path (20) edge node {} (22);
    \path (22) edge node {} (24);
    \path (40) edge node {} (42);
    \path (42) edge node {} (44);
    \path (60) edge node {} (62);
    \path (62) edge node {} (64);
    \path (00) edge node {} (11);
    \path (20) edge node {} (11);
    \path (02) edge node {} (11);
    \path (22) edge node {} (11);
    \path (20) edge node {} (31);
    \path (40) edge node {} (31);
    \path (22) edge node {} (31);
    \path (42) edge node {} (31);
    \path (40) edge node {} (51);
    \path (60) edge node {} (51);
    \path (42) edge node {} (51);
    \path (62) edge node {} (51);
    \path (02) edge node {} (13);
    \path (22) edge node {} (13);
    \path (04) edge node {} (13);
    \path (24) edge node {} (13);
    \path (22) edge node {} (33);
    \path (42) edge node {} (33);
    \path (24) edge node {} (33);
    \path (44) edge node {} (33);
    \path (42) edge node {} (53);
    \path (62) edge node {} (53);
    \path (44) edge node {} (53);
    \path (64) edge node {} (53);
    \path (11) edge node {} (31);
    \path (31) edge node {} (51);
    \path (13) edge node {} (33);
    \path (33) edge node {} (53);
    \path (11) edge node {} (13);
    \path (31) edge node {} (33);
    \path (51) edge node {} (53);
    \node at (0,-1) {};
    \node at (0,5) {};
    \end{tikzpicture}
    \qquad \qquad 
    \begin{tikzpicture}
    \node at (-1,0) {};
    \node at (5,0) {};
    \node[shape=circle,draw=black, fill=red] (00) at (0,0) {};
    \node[shape=circle,draw=black, fill=blue] (02) at (0,2) {};
    \node[shape=circle,draw=black, fill=blue] (22) at (2,2) {};
    \node[shape=circle,draw=black, fill=red] (40) at (4,0) {};
    \node[shape=circle,draw=black, fill=blue] (42) at (4,2) {};
    \node[shape=circle,draw=black, fill=red] (04) at (0,4) {};
    \node[shape=circle,draw=black, fill=red] (44) at (4,4) {};
    \node[shape=circle,draw=black, fill=blue] (11) at (1,1) {};
    \node[shape=circle,draw=black, fill=blue] (31) at (3,1) {};
    \node[shape=circle,draw=black, fill=blue] (13) at (1,3) {};
    \node[shape=circle,draw=black, fill=blue] (33) at (3,3) {};
    \path (02) edge node {} (22);
    \path (22) edge node {} (42);
    \path (00) edge node {} (02);
    \path (02) edge node {} (04);
    \path (40) edge node {} (42);
    \path (42) edge node {} (44);
    \path (00) edge node {} (11);
    \path (02) edge node {} (11);
    \path (22) edge node {} (11);
    \path (40) edge node {} (31);
    \path (22) edge node {} (31);
    \path (42) edge node {} (31);
    \path (02) edge node {} (13);
    \path (22) edge node {} (13);
    \path (04) edge node {} (13);
    \path (22) edge node {} (33);
    \path (42) edge node {} (33);
    \path (44) edge node {} (33);
    \path (11) edge node {} (31);
    \path (13) edge node {} (33);
    \path (11) edge node {} (13);
    \path (31) edge node {} (33);
    \node at (0,-1) {};
    \node at (0,5) {};
    \end{tikzpicture}
    \qquad \qquad 
    \begin{tikzpicture}
    \node at (-0.5,0) {};
    \node at (5.5,0) {};
    \node[shape=circle,draw=black, fill=red] (00) at (0,0) {};
    \node[shape=circle,draw=black, fill=blue] (20) at (2,0) {};
    \node[shape=circle,draw=black, fill=blue] (02) at (0,2) {};
    \node[shape=circle,draw=black, fill=blue] (22) at (2,2) {};
    \node[shape=circle,draw=black, fill=blue] (40) at (4,0) {};
    \node[shape=circle,draw=black, fill=blue] (42) at (4,2) {};
    \node[shape=circle,draw=black, fill=red] (60) at (5,-1) {};
    \node[shape=circle,draw=black, fill=red] (04) at (0,4) {};
    \node[shape=circle,draw=black, fill=blue] (24) at (2,4) {};
    \node[shape=circle,draw=black, fill=blue] (44) at (4,4) {};
    \node[shape=circle,draw=black, fill=red] (64) at (5,5) {};
    \node[shape=circle,draw=black, fill=blue] (11) at (1,1) {};
    \node[shape=circle,draw=black, fill=blue] (31) at (3,1) {};
    \node[shape=circle,draw=black, fill=blue] (51) at (5,1) {};
    \node[shape=circle,draw=black, fill=blue] (13) at (1,3) {};
    \node[shape=circle,draw=black, fill=blue] (33) at (3,3) {};
    \node[shape=circle,draw=black, fill=blue] (53) at (5,3) {};
    \path (00) edge node {} (20);
    \path (20) edge node {} (40);
    \path (40) edge node {} (60);
    \path (02) edge node {} (22);
    \path (22) edge node {} (42);
    \path (04) edge node {} (24);
    \path (24) edge node {} (44);
    \path (44) edge node {} (64);
    \path (00) edge node {} (02);
    \path (02) edge node {} (04);
    \path (20) edge node {} (22);
    \path (22) edge node {} (24);
    \path (40) edge node {} (42);
    \path (42) edge node {} (44);
    \path (00) edge node {} (11);
    \path (20) edge node {} (11);
    \path (02) edge node {} (11);
    \path (22) edge node {} (11);
    \path (20) edge node {} (31);
    \path (40) edge node {} (31);
    \path (22) edge node {} (31);
    \path (42) edge node {} (31);
    \path (40) edge node {} (51);
    \path (60) edge node {} (51);
    \path (42) edge node {} (51);
    \path (02) edge node {} (13);
    \path (22) edge node {} (13);
    \path (04) edge node {} (13);
    \path (24) edge node {} (13);
    \path (22) edge node {} (33);
    \path (42) edge node {} (33);
    \path (24) edge node {} (33);
    \path (44) edge node {} (33);
    \path (42) edge node {} (53);
    \path (44) edge node {} (53);
    \path (64) edge node {} (53);
    \path (11) edge node {} (31);
    \path (31) edge node {} (51);
    \path (13) edge node {} (33);
    \path (33) edge node {} (53);
    \path (11) edge node {} (13);
    \path (31) edge node {} (33);
    \path (51) edge node {} (53);
    \end{tikzpicture}}\\
    \phantom{-}\\
    \scalebox{0.4}{
    \begin{tikzpicture}
        \node[shape=circle,draw=black, fill=red] (00) at (0,0) {};
        \node[shape=circle,draw=black, fill=blue] (10) at (2,0) {};
        \node[shape=circle,draw=black, fill=blue] (01) at (0,2) {};
        \node[shape=circle,draw=black, fill=blue] (11) at (2,2) {};
        \node[shape=circle,draw=black, fill=blue] (11') at (2.5,2.5) {};
        \node[shape=circle,draw=black, fill=blue] (20) at (4,0) {};
        \node[shape=circle,draw=black, fill=blue] (21) at (4,2) {};
        
        \node[shape=circle,draw=black, fill=blue] (21') at (4.5,2.5) {};
        \node[shape=circle,draw=black, fill=red] (02) at (0,4) {};
        \node[shape=circle,draw=black, fill=blue] (12) at (2,4) {};
        \node[shape=circle,draw=black, fill=blue] (22) at (4,4) {};
        
        \node[shape=circle,draw=black, fill=red] (30) at (6,0) {};
        \node[shape=circle,draw=black, fill=blue] (31) at (6,2) {};
        \node[shape=circle,draw=black, fill=red] (32) at (6,4) {};
        
        \path (00) edge node {} (10);
        \path (20) edge node {} (10);
        \path (01) edge node {} (11);
        \path (21) edge node {} (11);
        \path (00) edge node {} (01);
        \path (10) edge node {} (11);
        \path (20) edge node {} (21);
        \path (20) edge node {} (30);
        \path (21) edge node {} (31);
        \path (30) edge node {} (31);
        \path (22) edge node {} (32);
        \path (32) edge node {} (31);
        
        \path (21') edge node {} (11');
        \path (22) edge node {} (21');
        \path (21') edge node {} (32);
        \path (21') edge node {} (31);
        
        \path (11) edge node {} (11');
        \path (21) edge node {} (21');
        \path (01) edge node {} (11');
        \path (10) edge node {} (11');
        \path (12) edge node {} (11');
        \path (21') edge node {} (20);
        
        \path (02) edge node {} (12);
        \path (12) edge node {} (11);
        \path (02) edge node {} (01);
        \path (21) edge node {} (22);
        \path (22) edge node {} (12);
        
        \path (22) edge node {} (11');
        
        \path (01) edge node {} (12);
        \path (00) edge node {} (11);
        \path (10) edge node {} (21);
        \path (20) edge node {} (31);
    \end{tikzpicture}
    \qquad \qquad 
        \begin{tikzpicture}
        \node[shape=circle,draw=black, fill=red] (00) at (0,0) {};
        \node[shape=circle,draw=black, fill=blue] (10) at (2,0) {};
        \node[shape=circle,draw=black, fill=blue] (01) at (0,2) {};
        \node[shape=circle,draw=black, fill=blue] (11) at (2,2) {}; {};
        \node[shape=circle,draw=black, fill=blue] (20) at (4,0) {};
        \node[shape=circle,draw=black, fill=blue] (21) at (4,2) {};
        
        \node[shape=circle,draw=black, fill=red] (02) at (0,4) {};
        \node[shape=circle,draw=black, fill=blue] (12) at (2,4) {};
        \node[shape=circle,draw=black, fill=blue] (22) at (4,4) {};
        
        \node[shape=circle,draw=black, fill=red] (30) at (6,0) {};
        \node[shape=circle,draw=black, fill=blue] (31) at (6,2) {};
        \node[shape=circle,draw=black, fill=red] (32) at (6,4) {};
        
        \path (00) edge node {} (10);
        \path (20) edge node {} (10);
        \path (01) edge node {} (11);
        \path (21) edge node {} (11);
        \path (00) edge node {} (01);
        \path (10) edge node {} (11);
        \path (20) edge node {} (21);
        \path (20) edge node {} (30);
        \path (21) edge node {} (31);
        \path (30) edge node {} (31);
        \path (22) edge node {} (32);
        \path (32) edge node {} (31);

        \path (02) edge node {} (12);
        \path (12) edge node {} (11);
        \path (02) edge node {} (01);
        \path (21) edge node {} (22);
        \path (22) edge node {} (12);
        
        \path (01) edge node {} (12);
        \path (00) edge node {} (11);
        \path (10) edge node {} (21);
        \path (20) edge node {} (31);
    \end{tikzpicture}
    \qquad \qquad 
    \begin{tikzpicture}
    \node[shape=circle,draw=black, fill=red] (00) at (0,0) {};
    \node[shape=circle,draw=black, fill=blue] (20) at (2,0) {};
    \node[shape=circle,draw=black, fill=blue] (02) at (0,2) {};
    \node[shape=circle,draw=black, fill=blue] (22) at (2,2) {};
    \node[shape=circle,draw=black, fill=blue] (40) at (4,0) {};
    \node[shape=circle,draw=black, fill=blue] (42) at (4,2) {};
    \node[shape=circle,draw=black, fill=red] (60) at (6,0) {};
    \node[shape=circle,draw=black, fill=blue] (62) at (6,2) {};
    \node[shape=circle,draw=black, fill=red] (04) at (0,4) {};
    \node[shape=circle,draw=black, fill=blue] (24) at (2,4) {};
    \node[shape=circle,draw=black, fill=blue] (44) at (4,4) {};
    \node[shape=circle,draw=black, fill=red] (64) at (6,4) {};
    \node[shape=circle,draw=black, fill=blue] (51) at (5,1) {};
    \node[shape=circle,draw=black, fill=blue] (13) at (1,3) {};
    \node[shape=circle,draw=black, fill=blue] (33) at (3,3) {};
    \path (00) edge node {} (20);
    \path (20) edge node {} (40);
    \path (40) edge node {} (60);
    \path (02) edge node {} (22);
    \path (22) edge node {} (42);
    \path (42) edge node {} (62);
    \path (04) edge node {} (24);
    \path (24) edge node {} (44);
    \path (44) edge node {} (64);
    \path (00) edge node {} (02);
    \path (02) edge node {} (04);
    \path (20) edge node {} (22);
    \path (22) edge node {} (24);
    \path (40) edge node {} (42);
    \path (42) edge node {} (44);
    \path (60) edge node {} (62);
    \path (62) edge node {} (64);
    \path (40) edge node {} (51);
    \path (60) edge node {} (51);
    \path (42) edge node {} (51);
    \path (62) edge node {} (51);
    \path (02) edge node {} (13);
    \path (22) edge node {} (13);
    \path (04) edge node {} (13);
    \path (24) edge node {} (13);
    \path (22) edge node {} (33);
    \path (42) edge node {} (33);
    \path (24) edge node {} (33);
    \path (44) edge node {} (33);
    \path (13) edge node {} (33);
    \end{tikzpicture}
    }
    \caption{The graphs $N_{3\times 2}$, $X_{3\times 2}$, and $T_{3\times 2}$ are shown in the top row. The graphs $D_{3\times 2}$, $L_{3\times 2}$, and a subgraph of $N_{3\times 2}$ induced by $V_{3\times 2}^0\cup (W\cap V_{3\times 2}^1)$ for some axis slice convex $W\subseteq \mathbb{R}^2$ are shown in the second row. Vertices in $(\partial G)'$ are shown in red. Observe that in each case, $\partial G\setminus (\partial G)'\neq \emptyset$}
    \label{graphexamples}
\end{figure}
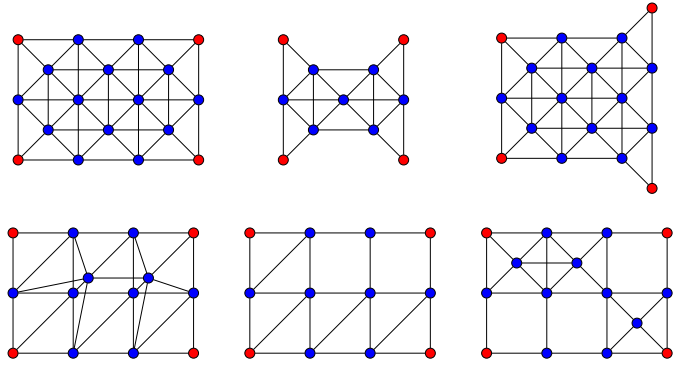

\begin{theorem}[M\"{u}ller, P\'{o}r, and Sereni \cite{mullerfour}, Theorem 4]\label{mullerthm}
A connected graph $G$ has $|(\partial G)'|=4$ if and only if $G$ is one of the following graphs:
\begin{enumerate}[label=(\alph*)]
    \item a subdivision of the star graph $K_{1,4}$,
    \item a subdivision of the tree with exactly four leaves and two vertices of degree~3,
    \item a graph obtained from one of the trees in (b) by removing a vertex of degree 3 and adding all edges between its neighbors,
    \item a subgraph of $N_{a\times c}$ induced by $V_{a\times c}^0\cup (W\cap V_{a\times c}^1)$ for some axis slice convex set $W\subseteq \mathbb{R}^2$, with exactly one path of arbitrary length attached to each of its CEJZ boundary vertices,
    \item the graph $X_{a\times c}$ with exactly one path of arbitrary length attached to each of its CEJZ boundary vertices,
    \item a subgraph of $T_{a\times c}$ induced by $V_{a\times (c+1)}^1\cup (W\cap V_{a\times (c+1)}^0)$ for some axis-slice convex set $W\subseteq \mathbb{R}^2$ that contains $(a,0)$ and $(a,c+1)$, with exactly one path of arbitrary length attached to each of its CEJZ boundary vertices, 
    \item the graph $D_{a\times c}$ with exactly one path of arbitrary length attached to each of its CEJZ boundary vertices, or
    \item the graph $L_{a\times c}$ with exactly one path of arbitrary length attached to each of its CEJZ boundary vertices.
\end{enumerate}
\end{theorem}

\begin{remark}\label{mullerremark}
M\"{u}ller, P\'{o}r, and Sereni show that the graphs in (d)-(h) without any attached paths have CEJZ boundary vertices given by the points with extreme $x$ and $y$ coordinates. For example, $N_{a\times c}$ has CEJZ boundary $\{(0,0),(0,c),(a,0),(a,c)\}$, and $T_{a\times c}$ has boundary $\{(1/2,1/2),(1/2,c+1/2),(a,0),(a,c+1)\}$. Attaching paths to the vertices $v\in (\partial G)'$ replaces $v$ with the new leaf produced from the attachment, and hence preserves the number of CEJZ boundary vertices.
\end{remark}

\section{The boundary stability number of a vertex}\label{stability}

In this section, we introduce a new parameter called the boundary stability number of a vertex. We will establish several results involving this parameter, some of which we will use for the proof of \cref{maintheorem}. We start with the following observation, which can be verified directly from definitions.

\begin{observation}\label{neighbordistance}
Let $G=(V,E)$ be a connected graph, and let $v\in V$. If $w~\in~N(v)$, then for any $u\in V$, $|d(v,u)-d(w,u)|\leq 1$. Additionally, if $v\neq u$, then there exists some $w\in N(v)$ such that $d(w,u)=d(v,u)-1$, as some neighbor of $v$ is on a shortest path from $v$ to $u$.
\end{observation}

With this observation in mind, notice that for graphs on at least two vertices, $u$ being a witness for $v\in \partial G$ corresponds to when there are more neighbors of $v$ at distance $d(v,u)-1$ than at distance $d(v,u)+1$. We define a new parameter that records precisely this difference.

\begin{definition}
For a  connected graph $G=(V,E)$ and $v,u\in V$, define
\[\beta_G(v,u)=\sum_{w\in N_G(v)}[d_G(v,u)-d_G(w,u)],\]
and define the {boundary stability number} of $v$ in $G$ to be $\beta_G(v)=\max_{u\in V}\beta_G(v,u)$. When the context is clear, we omit the subscript in $\beta_G$.
\end{definition}

Observe that $\beta(v)$ is always an integer, and it is straightforward to show that when a connected graph $G$ has at least two vertices, $\beta(v)\geq 1$ is equivalent to $v\in \partial G$. When this occurs, any $u\in V$ with $\beta(v,u)\geq 1$ is a witness for $v\in \partial G$. However, the exact value of $\beta(v)$ provides some additional information, which we will use later to consider the effect of certain graph operations. Note that if $G=(V,E)$ is the single vertex graph, then $\beta(v)=0$ for the unique $v\in V$. Since we use the convention that $\partial G=V$ when $|V|=1$, this is the only case in which a boundary vertex does not satisfy $\beta(v)\geq 1$. Examples of $\beta(v)$ are shown in \cref{stabilityexample}.

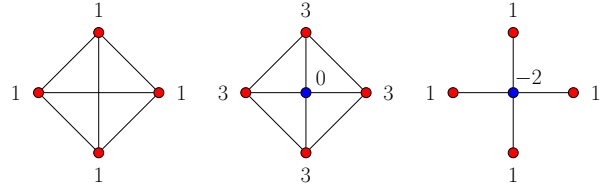
\begin{figure}[h!]
\centering
    \scalebox{0.4}{
    \begin{tikzpicture}
    \node[shape=circle,draw=black, fill=red] (A) at (0,2) {};
    \node[shape=circle,draw=black, fill=red] (B) at (-2,0) {};
    \node[shape=circle,draw=black, fill=red] (C) at (0,-2) {};
    \node[shape=circle,draw=black, fill=red] (D) at (2,0) {};
\node at (0,2.75) {\Huge{$1$}};
    \node at (2.75,0) {\Huge{$1$}};
    \node at (0,-2.75) {\Huge{$1$}};
    \node at (-2.75,0) {\Huge{$1$}};
    \path (A) edge node {} (B);
    \path (A) edge node {} (C);
    \path (A) edge node {} (D);
    \path (B) edge node {} (C);
    \path (B) edge node {} (D);
    \path (C) edge node {} (D);
\end{tikzpicture}
\qquad 
    \begin{tikzpicture}
    \node[shape=circle,draw=black, fill=red] (A) at (0,2) {};
    \node[shape=circle,draw=black, fill=red] (B) at (-2,0) {};
    \node[shape=circle,draw=black, fill=blue] (M) at (0,0) {};
    \node[shape=circle,draw=black, fill=red] (C) at (0,-2) {};
    \node[shape=circle,draw=black, fill=red] (D) at (2,0) {};
    \node at (0,2.75) {\Huge{$3$}};
    \node at (2.75,0) {\Huge{$3$}};
    \node at (0,-2.75) {\Huge{$3$}};
    \node at (-2.75,0) {\Huge{$3$}};
    \node at (0.5,0.5) {\Huge{$0$}};
    \path (A) edge node {} (M);
    \path (M) edge node {} (D);
    \path (A) edge node {} (D);
    \path (A) edge node {} (B);
    \path (B) edge node {} (C);
    \path (C) edge node {} (D);
    \path (B) edge node {} (M);
    \path (C) edge node {} (M);
\end{tikzpicture} \qquad 
\begin{tikzpicture}
    \node[shape=circle,draw=black, fill=red] (3) at (4,0) {};
    \node[shape=circle,draw=black, fill=blue] (4) at (6,0) {};
    \node[shape=circle,draw=black, fill=red] (5) at (8,0) {};
    \node[shape=circle,draw=black, fill=red] (6) at (6,2) {};
    \node[shape=circle,draw=black, fill=red] (7) at (6,-2) {};
    \path (3) edge node {} (4);
    \path (4) edge node {} (5);
    \path (4) edge node {} (6);
    \path (4) edge node {} (7);
    \node at (3.25,0) {\Huge{$1$}};
    \node at (6,2.75) {\Huge{$1$}};
    \node at (6.5,.5) {\Huge{$-2$}};
    \node at (6,-2.75) {\Huge{$1$}};
    \node at (8.75,0) {\Huge{$1$}};
    \end{tikzpicture}
    
   }
    \caption{Several graphs with the boundary stability number of each vertex shown next to it. In general, leaves in any graph and all vertices in the complete graph $K_n$ for $n\geq 2$ have boundary stability number 1.}
    \label{stabilityexample}
\end{figure}

We now establish a characterization of when vertices of degree 2 are in $\partial G$. Our method for this is based on part of the proof of \cite[Proposition 1.2]{chartrand}. 

\begin{lemma}\label{degree2lemma}
Let $G=(V,E)$ be a connected graph.  A vertex $v\in V$ with $\deg(v)=2$ is a boundary vertex if and only if there exists a cycle in $G$ that contains $v$.
\end{lemma}

\begin{proof}
Let $N(v)=\{v_1,v_2\}$.  Suppose there exists a cycle in $G$ that contains $v$. Then choose such a cycle $C$ so that the number of vertices in it is minimal and let $k$ denote the number of vertices in $C$. Observe that both neighbors of $v$ are in this cycle, and $d(v,w)\leq \lceil k/2\rceil$ for any $w\in C$. Since $C$ is a minimal cycle containing $v$, there exists $u\in C$ where equality holds. Additionally, $d(v_i,u)\leq \lceil k/2\rceil$ for $i\in \{1,2\}$, and combined with \cref{neighbordistance}, we see that \[\beta(v,u)=[d(v,u)-d(v_1,u)]+[d(v,u)-d(v_2,u)] \geq 1.\]
Hence, $\beta(v)\geq 1$ and $v\in \partial G$.

Conversely, suppose $v\in \partial G$. Fix a witness $u\in V$, and assume without loss of generality that some shortest path from $u$ to $v$ has the form $P_1=vv_1w_2\ldots w_{\ell-1}u$, so $d(v_1,u)=d(v,u)-1$. Combining \cref{neighbordistance} with $v\in \partial G$, it must be that $d(v_2,u)\leq d(v,u)$. This implies that a shortest path $P_2=v_2w_1'w_2'w_3'\ldots w_{k-1}'u$ from $v_2$ to $u$ cannot contain $v$. Then $W=vv_1w_2\ldots w_{\ell-1} uw_{k-1}'\ldots w_1'v_2v$ is a closed walk that contains $v,v_1$, and $v_2$. Finding the first vertex $w$ in $w_2w_3\ldots w_{\ell-1}u$ that also appears in $v_2w_1'w_2'w_3'\ldots w_{k-1}'u$ and removing all vertices in between them in $W$ produces a cycle containing $v$.
\end{proof}

It is straightforward to show that for a graph $G=(V,E)$, a vertex $v\in V$ of degree $2$ is either a cut vertex or contained in some cycle. Hence, we have the following alternative characterization for when $v$ is in $\partial G$.

\begin{corollary}\label{deg2cut}
Let $G=(V,E)$ be a connected graph. A vertex $v\in V$ with $\deg(v)=2$ is either a boundary vertex of $G$ or a cut vertex of $G$. 
\end{corollary}

We now consider the effects of certain graph operations. Observe that in \cref{mullerthm}, attaching paths to vertices in $(\partial G)'$ results in a graph $H$ with $|(\partial H)'|=|(\partial G)'|$. This is not always true for vertices in $\partial G$, and we will see that the boundary stability number allows us to determine precisely when the number of boundary vertices is preserved. We will use the following observation that is straightforward to establish.

\begin{observation}\label{cutdistance}
If $G=(V,E)$ is connected, $v$ is a cut vertex of $G$, and $u$ and $w$ are in different components of $G-v$, then $d(u,w)=d(u,v)+d(v,w)$.
\end{observation}

{
\begin{lemma}\label{graphedgeattachment}
Let $G_1=(V_1,E_1)$ be a connected graph on at least two vertices with $v_1\in V_1$, and $G_2=(V_2,E_2)$ be a connected graph where $V_2$ is disjoint from $V_1$. For any $v_2\in V_2$, define $G=(V_1\cup V_2,E_1\cup E_2\cup(v_1,v_2))$. Then $\beta_G(v_1)=\beta_{G_1}(v_1)-1$, and for $v\in V_1\setminus \{v_1\}$, we have $\beta_G(v)=\beta_{G_1}(v)$.
\end{lemma}

\begin{proof}
Observe that $v_1$ is a cut vertex of $G$, and $N_G(v_1)=N_{G_1}(v_1)\cup \{v_2\}$. Additionally, for $u,w\in V_1$, we have $d_G(w,u)=d_{G_1}(w,u)$. Combining this with \cref{cutdistance}, we see that for any $u\in V_1$,
\begin{equation*}
    \begin{split}
        \beta_G(v_1, u)&=\beta_{G_1}(v_1,u)+ [d_G(v_1,u)-d_G(v_2,u)] \\
    & = \beta_{G_1}(v_1,u)+[d_G(v_1,u)-(1+d_G(v_1,u))] \\
    &=\beta_{G_1}(v_1,u)-1.
    \end{split}
\end{equation*}
Additionally, for any $u\in V_2$, we again use \cref{cutdistance} to find
\begin{equation*}
    \begin{split}
        \beta_G(v_1, u) & = [d_G(v_1,u)-d_G(v_2,u)]+\sum_{w\in N_{G_1}(v)} [d_G(v_1,u)-d_G(w,u)]\\
    & = [d_G(v_1,u)-d_G(v_2,u)]+\sum_{w\in N_{G_1}(v_1)} [d_G(v_1,u)-(1+d_G(v_1,u))]\\
    & = 1-|N_{G_1}(v_1)|.
    \end{split}
\end{equation*}
Since $|V_1|\geq 2$, choosing some $u\in V_1$ distinct from $v_1$ and applying \cref{neighbordistance} implies that \[\beta_{G_1}(v_1)\geq \beta_G(v_1,u)\geq  1-(|N_{G_1}(v_1)|-1)=2-|N_{G_1}(v_1)|.\]
Hence, \[\beta_G(v_1)=\max_{u\in V_1\cup V_2} \{\beta_G(v_1,u)\}=\max_{u\in V_1}\{\beta_{G_1}(v_1,u)-1\}=\beta_{G_1}(v_1)-1.\]

Now consider $v\in V_1\setminus \{v_1\}$, where 
$N_{G}(v)=N_{G_1}(v)$. As before, for $u,w\in V_1$, we have $d_G(w,u)=d_{G_1}(w,u)$, so $\beta_G(v,u)=\beta_{G_1}(v,u)$. For $u\in V_2$, \cref{cutdistance} and $d_{G}(v,v_1)=d_{G_1}(v,v_1)$ imply
\begin{align*}
    \beta_G(v, u) & = \sum_{w\in N_{G_1}(v)} [d_G(v,u)-d_G(w,u)]\\
    & = \sum_{w\in N_{G_1}(v)} [(d_G(v,v_1)+d_G(v_1,u))-(d_G(w,v_1)+d_G(v_1,u))] \\
    & = \sum_{w\in N_{G_1}(v)} [d_{G}(v,v_1)-d_{G}(w,v_1)] \\
    & = \sum_{w\in N_{G_1}(v)} [d_{G_1}(v,v_1)-d_{G_1}(w,v_1)] \\
    & = \beta_{G_1}(v,v_1).
\end{align*}
Hence, we conclude that \[\beta_G(v)=\max_{u\in V_1\cup V_2} \{\beta_G(v,u)\}=\max_{u\in V_1}\{\beta_{G_1}(v,u)\}=\beta_{G_1}(v).\qedhere\]
\end{proof}

An example of \cref{graphedgeattachment} is shown in \cref{attachmentexample}. One consequence of this result is that we can determine precisely how $\partial G$ relates to $\partial G_1$ and $\partial G_2$. 

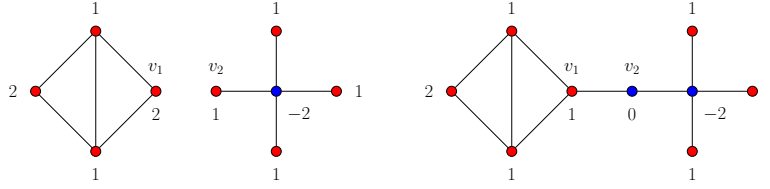
\begin{figure}[h!]
\begin{subfigure}{0.4\textwidth}
    \centering
    \scalebox{0.4}{\begin{tikzpicture}
    \node[shape=circle,draw=black, fill=red] (A) at (0,2) {};
    \node[shape=circle,draw=black, fill=red] (B) at (-2,0) {};
    \node[shape=circle,draw=black, fill=red] (C) at (0,-2) {};
    \node[shape=circle,draw=black, fill=red] (D) at (2,0) {};
    \node[shape=circle,draw=black, fill=red] (3) at (4,0) {};
    \node[shape=circle,draw=black, fill=blue] (4) at (6,0) {};
    \node[shape=circle,draw=black, fill=red] (5) at (8,0) {};
    \node[shape=circle,draw=black, fill=red] (6) at (6,2) {};
    \node[shape=circle,draw=black, fill=red] (7) at (6,-2) {};
    \path (A) edge node {} (B);
    \path (A) edge node {} (C);
    \path (A) edge node {} (D);
    \path (B) edge node {} (C);
    \path (C) edge node {} (D);
    \path (3) edge node {} (4);
    \path (4) edge node {} (5);
    \path (4) edge node {} (6);
    \path (4) edge node {} (7);
    \node at (-2.75,0) {\huge{$2$}};
    \node at (0,2.75) {\huge{$1$}};
    \node at (0,-2.75) {\huge{$1$}};
    \node at (2,-0.75) {\huge{$2$}};
    \node at (2,0.75) {\huge{$v_1$}};
    \node at (4,-0.75) {\huge{$1$}};
    \node at (4,0.75) {\huge{$v_2$}};
    \node at (6,2.75) {\huge{$1$}};
    \node at (6.75,-0.75) {\huge{$-2$}};
    \node at (6,-2.75) {\huge{$1$}};
    \node at (8.75,0) {\huge{$1$}};
    \end{tikzpicture}}
\end{subfigure}$~~~$
\begin{subfigure}{0.4\textwidth}
    \centering
    \scalebox{0.4}{
    \begin{tikzpicture}
    \node[shape=circle,draw=black, fill=red] (A) at (0,2) {};
    \node[shape=circle,draw=black, fill=red] (B) at (-2,0) {};
    \node[shape=circle,draw=black, fill=red] (C) at (0,-2) {};
    \node[shape=circle,draw=black, fill=red] (D) at (2,0) {};
    \node[shape=circle,draw=black, fill=blue] (3) at (4,0) {};
    \node[shape=circle,draw=black, fill=blue] (4) at (6,0) {};
    \node[shape=circle,draw=black, fill=red] (5) at (8,0) {};
    \node[shape=circle,draw=black, fill=red] (6) at (6,2) {};
    \node[shape=circle,draw=black, fill=red] (7) at (6,-2) {};
    \path (A) edge node {} (B);
    \path (A) edge node {} (C);
    \path (A) edge node {} (D);
    \path (B) edge node {} (C);
    \path (C) edge node {} (D);
    \path (D) edge node {} (3);
    \path (3) edge node {} (4);
    \path (4) edge node {} (5);
    \path (4) edge node {} (6);
    \path (4) edge node {} (7);
    \node at (-2.75,0) {\huge{$2$}};
    \node at (0,2.75) {\huge{$1$}};
    \node at (0,-2.75) {\huge{$1$}};
    \node at (2,-0.75) {\huge{$1$}};
    \node at (2,0.75) {\huge{$v_1$}};
    \node at (4,-0.75) {\huge{$0$}};
    \node at (4,0.75) {\huge{$v_2$}};
    \node at (6,2.75) {\huge{$1$}};
    \node at (6.75,-0.75) {\huge{$-2$}};
    \node at (6,-2.75) {\huge{$1$}};
    \node at (8.75,0) {\huge{$1$}};
    \end{tikzpicture}
    }
\end{subfigure}
\caption{Example of \cref{graphedgeattachment}. The graphs $G_1, G_2,$ and $G$ are given with boundary vertices shown in red. For each $v\in V_1,V_2,V$, $\beta(v)$ is shown next to $v$.}
\label{attachmentexample}
\end{figure}

\begin{theorem}\label{edgeattachment}
Let $G_1=(V_1,E_1)$ and $G_2=(V_2,E_2)$ be connected graphs on disjoint vertex sets, and let $v_1\in V_1$ and $v_2\in V_2$. Define $G=(V_1\cup V_2,E_1\cup E_2\cup(v_1,v_2))$.
\begin{enumerate}[label=(\alph*)]
            \item If $\beta_{G_1}(v_1)\neq 1$ and $\beta_{G_2}(v_2)\neq 1$, then $\partial G=\partial G_1\cup \partial G_2$.
        \item If $\beta_{G_1}(v_1)=1$ and $\beta_{G_2}(v_2)\neq 1$, then $\partial G=(\partial G_1\cup \partial G_2) \setminus \{v_1\}$.
        \item If $\beta_{G_1}(v_1)\neq 1$ and $\beta_{G_2}(v_2)= 1$, then $\partial G=(\partial G_1\cup \partial G_2) \setminus \{v_2\}$.
        \item If $\beta_{G_1}(v_1)=1$ and $\beta_{G_2}(v_2)=1$, , then $\partial G=(\partial G_1\cup \partial G_2) \setminus \{v_1,v_2\}$.
\end{enumerate}
In particular, we have  $|\partial G|\geq \max\{|\partial G_1|,|\partial G_2|\}$.
\end{theorem}
\begin{proof}
First, suppose that $|V_1|>1$ and $|V_2|>1$. Then \cref{graphedgeattachment} implies that for any vertex $v\in V_i\setminus \{v_i\}$, we have $v\in \partial G_i$ if and only if $v\in \partial G$. For $v_i$, observe that the same is true except when $\beta_{G_i}(v_i)=1$, and in this case $v_i\in \partial G_i$ and $v_i\notin \partial G$. The theorem then follows from the various cases.

Now suppose $|V_1|>1$ and $|V_2|=1$. In this case, $\partial G_2=\{v_2\}$ and $\beta_{G_2}(v_2)=0$, so it suffices to show (a) and (b) with $\partial G_2=\{v_2\}$. Since $v_2$ is a leaf in $G$, \cref{treeandleaves} implies $v_2\in \partial G$. Applying the above argument with \cref{graphedgeattachment} on $V_1$, we conclude (a) and (b). The case $|V_1|=1$ and $|V_2|>1$ is similar. Finally, when $|V_1|=|V_2|=1$, we see that $\partial G_i=\{v_i\}$ and $\beta_{G_i}(v_i)=0$ for all $i$. In this case, $G$ is a path graph on two vertices, and (a) is clear.
\end{proof}
}

We conclude this section with results on subgraphs. Using the boundary stability number, we can sometimes determine when boundary vertices in a subgraph of $G$ are also boundary vertices in $G$ itself.

\begin{lemma}\label{boundarysubgraph}
Let $G$ be a connected graph with a connected subgraph $H$ on at least two vertices. Suppose $v\in \partial H$, and let $u$ be a witness for $v$. If $N_{H}(v)=N_G(v)$ and $d_{H}(v,u)=d_G(v,u)$, then $v \in \partial G$.
\end{lemma}

\begin{proof}
Notice that for any $w\in N_G(v)=N_{H}(v)$, we have $d_G(w,u)\leq d_{H}(w,u)$. Then a direct calculation shows 
    \begin{align*}
        \beta_{G}(v)&\geq \beta_G(v,u) \\
        & =\sum_{w\in N_G(v)}[d_G(v,u)-d_G(w,u)] \\
        & \geq \sum_{w\in N_{H}(v)}[d_{H}(v,u)-d_{H}(w,u)]\\
        & = \beta_{H}(v,u) \\
        & \geq 1. \qedhere
    \end{align*}
\end{proof}

\begin{corollary}\label{boundarysubgraphcor}
Let $G$ be a connected graph with a connected subgraph $H$ on at least two vertices. Suppose $v\in \partial H$, and let $u$ be a witness for $v$. If $N_{H}(v)=N_G(v)$ and $d_{H}(v,u)=2$, then $v \in \partial G$.
\end{corollary}
\begin{proof}
Since $H$ is a subgraph of $G$, we have $d_G(v,u)\leq d_H(v,u)=2$. Observe that $d_H(v,u)=2$ implies $u\notin N_H(v)$, and the assumption $N_{H}(v)=N_G(v)$ implies $u\notin N_G(v)$. Then $d_G(v,u)>1$, and combined, we conclude that $d_G(v,u)=2$. The result then follows from \cref{boundarysubgraph}.
\end{proof}

\section{Graphs with at most four boundary vertices}\label{boundarysize4}

In this section, we establish the proof of \cref{maintheorem}. We start by applying the results of \cref{stability} to establish lemmas for graphs with $|\partial G|=3$ or $|\partial G|=4$.

\begin{lemma}\label{boundary3lemma}
Let $G=(V,E)$ be a connected graph. If $|(\partial G)'|=3$, then $|\partial G|=3$. 
\end{lemma}

\begin{proof}
We use the characterization of $|(\partial G)'|=3$ given in \cref{CEJZ3}. If $G$ is a tree on three leaves, then \cref{treeandleaves} implies that $|\partial G|=3$. For tripods, we start by considering the complete graph $K_3$. A direct calculation shows that $|\partial K_3|=3$, and each vertex $v\in V$ has $\beta(v)=1$. Tripods are formed by attaching arbitrary length paths to the vertices of $K_3$. By applying \cref{graphedgeattachment} for each nontrivial path attached, we conclude that tripods have three boundary vertices. 
\end{proof}

\begin{lemma}\label{boundary4lemma}
Let $G=(V,E)$ be a connected graph with $|(\partial G)'|=4$. Then $|\partial G|=4$ if and only if $G$ is one of the following graphs:
\begin{enumerate}[label=(\alph*)]
    \item a subdivision of the star graph $K_{1,4}$,
    \item a subdivision of the tree with exactly four leaves and two vertices of degree~3,
    \item a graph obtained from one of the trees in (b) by removing a vertex of degree 3 and adding edges between all of its neighbors, or
    \item one of the graphs in \cref{partialG4base} with a path of arbitrary length attached to each $v\in \partial G$ with $\beta(v)=1$. 
\end{enumerate}
\end{lemma}

\begin{figure}[h!]
    \centering
    \scalebox{0.4}{
\begin{tikzpicture}
    \node[shape=circle,draw=black, fill=red] (A) at (0,2) {};
    \node[shape=circle,draw=black, fill=red] (B) at (-2,0) {};
    \node[shape=circle,draw=black, fill=blue] (M) at (0,0) {};
    \node[shape=circle,draw=black, fill=red] (C) at (0,-2) {};
    \node[shape=circle,draw=black, fill=red] (D) at (2,0) {};

    \path (A) edge node {} (M);
    \path (M) edge node {} (D);
    \path (A) edge node {} (D);
    \path (A) edge node {} (B);
    \path (B) edge node {} (C);
    \path (C) edge node {} (D);
    \path (B) edge node {} (M);
    \path (C) edge node {} (M);
    \node at (0,-3) {\Huge{$N_{1,1}$}};
\end{tikzpicture}\qquad \qquad 
\begin{tikzpicture}
    \node[shape=circle,draw=black, fill=red] (A) at (0,2) {};
    \node[shape=circle,draw=black, fill=red] (B) at (-2,0) {};
    \node[shape=circle,draw=black, fill=red] (C) at (0,-2) {};
    \node[shape=circle,draw=black, fill=red] (D) at (2,0) {};

    \path (A) edge node {} (B);
    \path (A) edge node {} (D);
    \path (B) edge node {} (C);
    \path (C) edge node {} (D);
    \node at (0,-3) {\Huge{$C_4$}};
\end{tikzpicture}\qquad \qquad 
\begin{tikzpicture}
    \node[shape=circle,draw=black, fill=red] (A) at (0,2) {};
    \node[shape=circle,draw=black, fill=red] (B) at (-2,0) {};
    \node[shape=circle,draw=black, fill=red] (C) at (0,-2) {};
    \node[shape=circle,draw=black, fill=red] (D) at (2,0) {};

    \path (A) edge node {} (B);
    \path (A) edge node {} (C);
    \path (A) edge node {} (D);
    \path (B) edge node {} (C);
    \path (B) edge node {} (D);
    \path (C) edge node {} (D);
    \node at (0,-3) {\Huge{$K_4=X_{1,1}$}};
\end{tikzpicture}} 
\\
\phantom{-}\\
\scalebox{0.4}{
\begin{tikzpicture}
    \node[shape=circle,draw=black, fill=red] (A) at (-4,2) {};
    \node[shape=circle,draw=black, fill=red] (D) at (-4,-2) {};
    \node[shape=circle,draw=black, fill=blue] (M1) at (-2,0) {};
    \node (P) at (0,0) {$\cdots$};
    \node[shape=circle,draw=black, fill=blue] (M2) at (2,0) {};
    \node[shape=circle,draw=black, fill=red] (C) at (4,-2) {};
    \node[shape=circle,draw=black, fill=red] (B) at (4,2) {};

    \path (A) edge node {} (M1);
    \path (M1) edge node {} (D);
    \path (A) edge node {} (D);
    \path (B) edge node {} (C);
    \path (B) edge node {} (M2);
    \path (C) edge node {} (M2);
    \path (M1) edge node {} (P);
    \path (P) edge node {} (M2);
    \node at (0,-3) {\Huge{$X_{1\times c}$}};
\end{tikzpicture}\qquad \qquad 
\begin{tikzpicture}
    \node[shape=circle,draw=black, fill=red] (A) at (0,2) {};
    \node[shape=circle,draw=black, fill=red] (B) at (-2,0) {};
    \node[shape=circle,draw=black, fill=blue] (M) at (0,0) {};
    \node[shape=circle,draw=black, fill=red] (C) at (0,-2) {};
    \node[shape=circle,draw=black, fill=red] (D) at (2,0) {};

    \path (A) edge node {} (M);
    \path (M) edge node {} (D);
    \path (A) edge node {} (D);
    \path (A) edge node {} (B);
    \path (B) edge node {} (C);
    \path (B) edge node {} (M);
    \path (C) edge node {} (M);
    \node at (0,-3) {\Huge{$T_{1,1}$}};
\end{tikzpicture}\qquad \qquad 
\begin{tikzpicture}
    \node[shape=circle,draw=black, fill=red] (A) at (0,2) {};
    \node[shape=circle,draw=black, fill=red] (B) at (-2,0) {};
    \node[shape=circle,draw=black, fill=red] (C) at (0,-2) {};
    \node[shape=circle,draw=black, fill=red] (D) at (2,0) {};

    \path (A) edge node {} (B);
    \path (A) edge node {} (C);
    \path (A) edge node {} (D);
    \path (B) edge node {} (C);
    \path (C) edge node {} (D);
    \node at (0,-3) {\Huge{$D_{1,1}=L_{1,1}$}};
\end{tikzpicture}}

\caption{Graphs from \cref{mullerthm} (d)-(h) with $|\partial G|=4$.}
\label{partialG4base}
\end{figure}
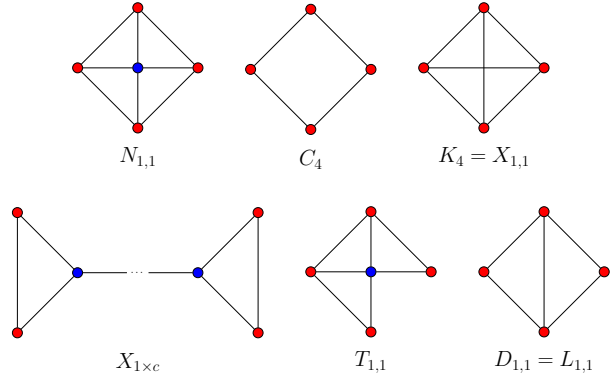

\begin{proof}
We show this by considering each case of \cref{mullerthm}. Note that graphs (a), (b), and (c) match \cref{mullerthm}, and the results for (a) and (b) follow directly from \cref{treeandleaves}. For (c), notice that these graphs $G$ can alternatively be constructed by starting with a tripod $H$ and attaching two nontrivial paths to some $v\in \partial H$. Observe that $H$ is constructed by attaching paths of arbitrary lengths to $K_3$. Using \cref{graphedgeattachment} on $K_3$ with paths attached, the boundary stability number of any vertex in $\partial H$ is 1. Using \cref{graphedgeattachment} once for each path attached to $v\in \partial H$, we conclude that $|\partial G|=4$.

Now consider the remaining cases in \cref{mullerthm}. Notice that if $G$ is constructed as a subgraph of $N_{a\times c}$, $X_{a\times c}$, $T_{a\times c}$, $D_{a\times c}$, or $L_{a\times c}$ and $|\partial G|\geq |(\partial G)'|\geq 5$, then by \cref{edgeattachment}, a graph $H$ formed by attaching paths to $G$ has $|\partial H|\geq |\partial G|\geq 5$. Hence, we must consider constructions where $|(\partial G)'|=4$ to obtain a graph $H$ with $|\partial H|=4$. We do this using the graphs in \cref{basecases}, which will allow us to show the existence of additional Steinerberger boundary vertices in addition to the four CEJZ boundary vertices described in \cref{mullerremark}. Notice that for each $G_i$ in the figure, $v\in \partial G_i$ with witness $u$, and $d_{G_i}(v,u)=2$.

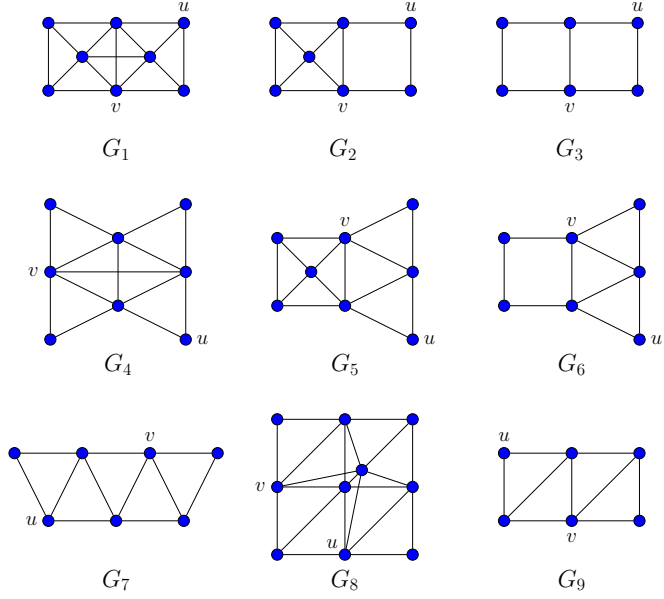
\begin{figure}[h!]
    \centering
    \scalebox{0.45}{
    \begin{tikzpicture}
    \node at (-1,0) {};
    \node at (5,0) {};
    \node[shape=circle,draw=black, fill=blue] (00) at (0,0) {};
    \node[shape=circle,draw=black, fill=blue] (10) at (2,0) {};
    \node[shape=circle,draw=black, fill=blue] (01) at (0,2) {};
    \node[shape=circle,draw=black, fill=blue] (11) at (2,2) {};
    \node[shape=circle,draw=black, fill=blue] (20) at (4,0) {};
    \node[shape=circle,draw=black, fill=blue] (21) at (4,2) {};
    \path (00) edge node {} (10);
    \path (20) edge node {} (10);
    \path (01) edge node {} (11);
    \path (21) edge node {} (11);
    \path (00) edge node {} (01);
    \path (10) edge node {} (11);
    \path (20) edge node {} (21);
    \node[shape=circle,draw=black, fill=blue] (05) at (1,1) {};
    \node[shape=circle,draw=black, fill=blue] (15) at (3,1) {};
    \path (00) edge node {} (05);
    \path (10) edge node {} (05);
    \path (01) edge node {} (05);
    \path (11) edge node {} (05);
    \path (10) edge node {} (15);
    \path (11) edge node {} (15);
    \path (20) edge node {} (15);
    \path (21) edge node {} (15);
    \path (05) edge node {} (15);
    \node at (2,-0.5) {\huge{$v$}};
    \node at (4,2.5) {\huge{$u$}};
    \node at (2,-1.75) {\Huge{$G_1$}};
    \end{tikzpicture}
    \quad 
    \begin{tikzpicture}
    \node at (-1,0) {};
    \node at (5,0) {};
    \node[shape=circle,draw=black, fill=blue] (00) at (0,0) {};
    \node[shape=circle,draw=black, fill=blue] (10) at (2,0) {};
    \node[shape=circle,draw=black, fill=blue] (01) at (0,2) {};
    \node[shape=circle,draw=black, fill=blue] (11) at (2,2) {};
    \node[shape=circle,draw=black, fill=blue] (20) at (4,0) {};
    \node[shape=circle,draw=black, fill=blue] (21) at (4,2) {};
    \path (00) edge node {} (10);
    \path (20) edge node {} (10);
    \path (01) edge node {} (11);
    \path (21) edge node {} (11);
    \path (00) edge node {} (01);
    \path (10) edge node {} (11);
    \path (20) edge node {} (21);
    \node[shape=circle,draw=black, fill=blue] (05) at (1,1) {};
    \path (00) edge node {} (05);
    \path (10) edge node {} (05);
    \path (01) edge node {} (05);
    \path (11) edge node {} (05);
     \node at (2,-0.5) {\huge{$v$}};
    \node at (4,2.5) {\huge{$u$}};
    \node at (2,-1.75) {\Huge{$G_2$}};
    \node at (-0.5,2) {};
    \end{tikzpicture}
    \quad 
    \begin{tikzpicture}
    
    \node at (-1,0) {};
    \node at (5,0) {};
    \node[shape=circle,draw=black, fill=blue] (00) at (0,0) {};
    \node[shape=circle,draw=black, fill=blue] (10) at (2,0) {};
    \node[shape=circle,draw=black, fill=blue] (01) at (0,2) {};
    \node[shape=circle,draw=black, fill=blue] (11) at (2,2) {};
    \node[shape=circle,draw=black, fill=blue] (20) at (4,0) {};
    \node[shape=circle,draw=black, fill=blue] (21) at (4,2) {};
    \path (00) edge node {} (10);
    \path (20) edge node {} (10);
    \path (01) edge node {} (11);
    \path (21) edge node {} (11);
    \path (00) edge node {} (01);
    \path (10) edge node {} (11);
    \path (20) edge node {} (21);
     \node at (2,-0.5) {\huge{$v$}};
    \node at (4,2.5) {\huge{$u$}};
    \node at (2,-1.75) {\Huge{$G_3$}};
    \end{tikzpicture} 
    }
    \\
    \phantom{-}\\
    \scalebox{0.45}{
    \begin{tikzpicture}
    \node at (-1,0) {};
    \node at (5,0) {};
    \node[shape=circle,draw=black, fill=blue] (00) at (0,0) {};
    \node[shape=circle,draw=black, fill=blue] (10) at (0,2) {};
    \node[shape=circle,draw=black, fill=blue] (01) at (4,0) {};
    \node[shape=circle,draw=black, fill=blue] (11) at (4,2) {};
    \node[shape=circle,draw=black, fill=blue] (20) at (0,4) {};
    \node[shape=circle,draw=black, fill=blue] (21) at (4,4) {};
    \path (00) edge node {} (10);
    \path (20) edge node {} (10);
    \path (01) edge node {} (11);
    \path (21) edge node {} (11);
    \path (10) edge node {} (11);
    \node[shape=circle,draw=black, fill=blue] (05) at (2,1) {};
    \node[shape=circle,draw=black, fill=blue] (15) at (2,3) {};
    \path (00) edge node {} (05);
    \path (10) edge node {} (05);
    \path (01) edge node {} (05);
    \path (11) edge node {} (05);
    \path (10) edge node {} (15);
    \path (11) edge node {} (15);
    \path (20) edge node {} (15);
    \path (21) edge node {} (15);
    \path (05) edge node {} (15);
    \node at (-0.5,2) {\huge{$v$}};
    \node at (4.5,0) {\huge{$u$}};
    \node at (2,-0.75) {\Huge{$G_4$}};
    \end{tikzpicture}
    \quad 
    \begin{tikzpicture}
    \node at (-1,0) {};
    \node at (5,0) {};
    \node[shape=circle,draw=black, fill=blue] (00) at (0,0) {};
    \node[shape=circle,draw=black, fill=blue] (02) at (0,2) {};
    \node[shape=circle,draw=black, fill=blue] (20) at (2,0) {};
    \node[shape=circle,draw=black, fill=blue] (22) at (2,2) {};
    \node[shape=circle,draw=black, fill=blue] (4n) at (4,-1) {};
    \node[shape=circle,draw=black, fill=blue] (41) at (4,1) {};
    \node[shape=circle,draw=black, fill=blue] (43) at (4,3) {};
    \node[shape=circle,draw=black, fill=blue] (11) at (1,1) {};
    \path (00) edge node {} (20);
    \path (00) edge node {} (02);
    \path (02) edge node {} (22);
    \path (22) edge node {} (20);
    \path (22) edge node {} (43);
    \path (22) edge node {} (41);
    \path (20) edge node {} (41);
    \path (43) edge node {} (41);
    \path (20) edge node {} (4n);
    \path (41) edge node {} (4n);
    \path (11) edge node {} (00);
    \path (11) edge node {} (20);
    \path (11) edge node {} (02);
    \path (11) edge node {} (22);
    \node at (2,2.5) {\huge{$v$}};
    \node at (4.5,-1) {\huge{$u$}};
    \node at (2,-1.75) {\Huge{$G_5$}};
    \end{tikzpicture}
    \quad 
    \begin{tikzpicture}
    \node at (-1,0) {};
    \node at (5,0) {};
    \node[shape=circle,draw=black, fill=blue] (00) at (0,0) {};
    \node[shape=circle,draw=black, fill=blue] (02) at (0,2) {};
    \node[shape=circle,draw=black, fill=blue] (20) at (2,0) {};
    \node[shape=circle,draw=black, fill=blue] (22) at (2,2) {};
    \node[shape=circle,draw=black, fill=blue] (4n) at (4,-1) {};
    \node[shape=circle,draw=black, fill=blue] (41) at (4,1) {};
    \node[shape=circle,draw=black, fill=blue] (43) at (4,3) {};
    \path (00) edge node {} (20);
    \path (00) edge node {} (02);
    \path (02) edge node {} (22);
    \path (22) edge node {} (20);
    \path (22) edge node {} (43);
    \path (22) edge node {} (41);
    \path (20) edge node {} (41);
    \path (43) edge node {} (41);
    \path (20) edge node {} (4n);
    \path (41) edge node {} (4n);
    \node at (2,2.5) {\huge{$v$}};
    \node at (4.5,-1) {\huge{$u$}};
    \node at (2,-1.75) {\Huge{$G_6$}};
    \end{tikzpicture}}
    \\
     \phantom{-}\\
     \scalebox{0.45}{
    \begin{tikzpicture}
    \node[shape=circle,draw=black, fill=blue] (00) at (0,0) {};
    \node[shape=circle,draw=black, fill=blue] (02) at (2,0) {};
    \node[shape=circle,draw=black, fill=blue] (04) at (4,0) {};
    \node[shape=circle,draw=black, fill=blue] (1n) at (-1,2) {};
    \node[shape=circle,draw=black, fill=blue] (11) at (1,2) {};
    \node[shape=circle,draw=black, fill=blue] (13) at (3,2) {};
    \node[shape=circle,draw=black, fill=blue] (15) at (5,2) {};
    \path (00) edge node {} (02);
    \path (00) edge node {} (1n);
    \path (00) edge node {} (11);
    \path (02) edge node {} (11);
    \path (02) edge node {} (13);
    \path (02) edge node {} (04);
    \path (04) edge node {} (13);
    \path (04) edge node {} (15);
    \path (13) edge node {} (15);
    \path (1n) edge node {} (11);
    \path (11) edge node {} (13);
    \node at (3,2.5) {\huge{$v$}};
    \node at (-0.5,0) {\huge{$u$}};
    \node at (2,-1.75) {\Huge{$G_7$}};
    \end{tikzpicture}
    \quad 
    \begin{tikzpicture}
    \node at (-1,0) {};
    \node at (5,0) {};
        \node[shape=circle,draw=black, fill=blue] (00) at (0,0) {};
        \node[shape=circle,draw=black, fill=blue] (10) at (2,0) {};
        \node[shape=circle,draw=black, fill=blue] (01) at (0,2) {};
        \node[shape=circle,draw=black, fill=blue] (11) at (2,2) {};
        \node[shape=circle,draw=black, fill=blue] (11') at (2.5,2.5) {};
        \node[shape=circle,draw=black, fill=blue] (20) at (4,0) {};
        \node[shape=circle,draw=black, fill=blue] (21) at (4,2) {};
        \node[shape=circle,draw=black, fill=blue] (02) at (0,4) {};
        \node[shape=circle,draw=black, fill=blue] (12) at (2,4) {};
        \node[shape=circle,draw=black, fill=blue] (22) at (4,4) {};
        \path (00) edge node {} (10);
        \path (20) edge node {} (10);
        \path (01) edge node {} (11);
        \path (21) edge node {} (11);
        \path (00) edge node {} (01);
        \path (10) edge node {} (11);
        \path (20) edge node {} (21);
        \path (11) edge node {} (11');
        \path (01) edge node {} (11');
        \path (10) edge node {} (11');
        \path (12) edge node {} (11');
        \path (02) edge node {} (12);
        \path (12) edge node {} (11);
        \path (02) edge node {} (01);
        \path (21) edge node {} (22);
        \path (22) edge node {} (12);
        \path (22) edge node {} (11');
        \path (01) edge node {} (12);
        \path (00) edge node {} (11);
        \path (10) edge node {} (21);
        \path (11') edge node {} (21);
        \node at (-0.5,2) {\huge{$v$}};
        \node at (1.625,0.325) {\huge{$u$}};
        \node at (2,-0.75) {\Huge{$G_8$}};
    \end{tikzpicture}
    \quad 
    \begin{tikzpicture}
    \node at (-1,0) {};
    \node at (5,0) {};
    \node[shape=circle,draw=black, fill=blue] (00) at (0,0) {};
    \node[shape=circle,draw=black, fill=blue] (10) at (2,0) {};
    \node[shape=circle,draw=black, fill=blue] (01) at (0,2) {};
    \node[shape=circle,draw=black, fill=blue] (11) at (2,2) {};
    \node[shape=circle,draw=black, fill=blue] (20) at (4,0) {};
    \node[shape=circle,draw=black, fill=blue] (21) at (4,2) {};
    \path (00) edge node {} (10);
    \path (20) edge node {} (10);
    \path (01) edge node {} (11);
    \path (21) edge node {} (11);
    \path (00) edge node {} (01);
    \path (10) edge node {} (11);
    \path (20) edge node {} (21);
    \path (00) edge node {} (11);
    \path (10) edge node {} (21);
    \node at (2,-0.5) {\huge{$v$}};
    \node at (0,2.5) {\huge{$u$}};
    \node at (2,-1.75) {\Huge{$G_9$}};
    \end{tikzpicture}
    }
    \caption{In each $G_i$, observe that $v\in \partial G_i$ with witness $u$, and $d_{G_i}(v,u)=2$.}
    \label{basecases}
\end{figure}

First, consider $N_{a\times c}$, and suppose $G$ is the subgraph induced by $V_{a\times c}^0\cup (W\cap V_{a\times c}^1)$ for axis slice convex $W$. 
If $a>1$ and $W$ contains both, one, or none of the vertices in $\{(1/2,1/2),(3/2,1/2)\}$, then $G$ will respectively contain $G_1$, $G_2$, or $G_3$ from \cref{basecases} as a subgraph with $N_G(v)=N_{G_i}(v)$, where in each case, $v$ corresponds to the vertex $(1,0)$ in $N_{a\times c}$. A similar argument applies if $c>1$, where $v$ corresponds to the vertex $(0,1)$.
In these cases, \cref{boundarysubgraphcor} implies that $v\in \partial G$ and $|\partial G|\geq 5$. From this, we see that the only graphs $G$ that have four boundary vertices must be constructed using $N_{1\times 1}$, which are $N_{1\times 1}$ itself or the cycle graph $C_4$. A direct verification shows $|\partial G|=4$ for both of these graphs, and they are shown in \cref{partialG4base}.

Now consider $G=X_{a\times c}$. Recall that $X_{a\times c}$ is isomorphic to $X_{c\times a}$, so assume without loss of generality that $a\leq c$. 
If $c\geq 2$ and $a=2$, then $X_{a\times c}$ contains $G_4$ as a subgraph with $N_G(v)=N_{G_4}(v)$, where $v$ corresponds to the vertex $(0,1)$ in $X_{a\times c}$. If $c\geq a\geq 2$, then $X_{a\times c}$ contains $G_5$ as a subgraph with $N_G(v)=N_{G_5}(v)$, where $v$ corresponds to the vertex at $(1/2,1/2)$. In these cases, \cref{boundarysubgraphcor} implies $|\partial G|\geq 5$. The graphs $X_{1\times 1}=K_4$ and $X_{1\times c}$ are depicted in \cref{partialG4base}, and a direct calculation shows that when $G$ is one of these graphs, $|\partial G|=4$.

Next, consider $T_{a\times c}$, and suppose that $G$ is the subgraph of $T_{a\times c}$ induced by $V_{a\times (c+1)}^1\cup (W\cap V_{a\times (c+1)}^0)$ for some axis slice convex $W$ that contains $(a,0)$ and $(a,c+1)$. If $a>1$ and $(a-1,c)\in W$, then $G$ will contain $G_5$ from \cref{basecases} as a subgraph with $N_G(v)=N_{G_5}(v)$, where $v$ corresponds to $(a-1/2,c+1/2)$. If $a>1$ and $(a-1,c)\not\in W$, then a similar statement is true for $G_6$, where $v$ again corresponds to $(a-1/2,c+1/2)$. If $a=1$ and $c>1$, then $G$ contains $G_7$ from \cref{basecases} as a subgraph with $N_G(v)=N_{G_7}(v)$, where $v$ corresponds to $(1,1)$. In all of these cases, \cref{boundarysubgraphcor} implies that $|\partial G|\geq 5$. Finally, when $a=c=1$,
note that any axis slice convex set containing $(1,0)$ and $(1,2)$ must also contain $(1,1)$. Hence, the only graph in this case is $T_{1\times 1}$. Direct verification shows $|\partial T_{1\times 1}|=4$, and this graph is shown in \cref{partialG4base}.

Now consider $G=D_{a\times c}$. If $a\geq 2$ and $c\geq 2$, then the graph $G_8$ from \cref{basecases} is a subgraph with $N_G(v)=N_{G_8}(v)$, where $v$ corresponds to the $(0,1)$ in $D_{a\times c}$.
If either $a=1$ and $c\geq 2$, or $a\geq 2$ and $c=1$, then $G_9$ from \cref{basecases} is a subgraph of $G$ with $N_G(v)=N_{G_9}(v)$. Thus in each of these cases, \cref{boundarysubgraphcor} implies $v\in \partial G$ and $|\partial G|\geq 5$. Finally, if $a=c=1$ then $|\partial G|=4$, and $D_{1\times 1}$ is depicted in \cref{partialG4base}.

Finally, consider $G=L_{a\times c}$. If $a>1$, then $G$ contains $G_{9}$ as a subgraph with $N_G(v)=N_{G_{9}}(v)$, where $v$ corresponds to $(1,0)$. When $a=1$, observe that by definition, $D_{1\times c}=L_{1\times c}$, and we have already considered these cases above.

Combined, we see that the only graphs described in \cref{mullerthm} parts (d)-(h) with $|\partial G|=4$ are those formed by attaching paths to the graphs in \cref{partialG4base} at the vertices in $\partial G=(\partial G)'$. By applying \cref{graphedgeattachment} repeatedly, we conclude that attaching nontrivial paths at $v\in \partial G$ preserves the number of boundary vertices if and only if $\beta(v)=~1$.
\end{proof}

We are now able to prove our main theorem characterizing graphs with small Steinerberger boundary.

\begin{proof}[Proof of \cref{maintheorem}]  
By \cref{CEJZ2} and \cref{boundary2}, $|\partial G|=2$ and $|(\partial G')|=2$ are equivalent, corresponding precisely to paths. This establishes \cref{maintheorem}(a). Combined with $(\partial G)'\subseteq \partial G$, we also see that a necessary condition for $|\partial G|=3$ is that $|(\partial G)'|=3$. Applying \cref{boundary3lemma}, we conclude that $|\partial G|=3$ if and only if $|(\partial G)'|=3$, implying \cref{maintheorem}(b). 

Using similar reasoning, we see that a necessary condition for $|\partial G|=4$ is that $|(\partial G)'|=4$. We consider each case in \cref{boundary4lemma}. By \cref{treeandleaves}, any tree on four leaves has $|\partial G|=4$, and this accounts for \cref{boundary4lemma} (a) and (b). For \cref{boundary4lemma} (c), observe that this is the last graph in \cref{partialG4}. The remaining graphs in \cref{boundary4lemma} (d) are also given in \cref{partialG4}. Hence, we conclude \cref{maintheorem}(c). 
\end{proof}

\section{Some graphs with large boundary}\label{largeboundary}

In this section, we consider some graphs with large Steinerberger boundary. Observe that the cycle and complete graphs consist entirely of boundary vertices, as each vertex is peripheral. In the case of $\diam(G)=2$, Chartrand, Erwin, Johns, and Zhang showed the following result on the CEJZ boundary. 

\begin{lemma}[Chartrand, Erwin, Johns, and Zhang \cite{chartrand}, Lemma 2.1]\label{diameter2}
Let $G=(V,E)$ be connected graph of diameter 2. Then every vertex $v$ is in $(\partial G)'$ unless $v$ is the unique vertex of $G$ having eccentricity $1$.
\end{lemma}

The same result holds for the Steinerberger boundary. We now show this, with some additional results.

\begin{theorem}\label{diameter}
Let $G=(V,E)$ be a connected graph on at least two vertices with $\diam(G)\leq 2$. If $G$ has a single vertex $v$ with eccentricity 1, then $\partial G=V\setminus \{v\}$. Otherwise, $\partial G=V$. Furthermore, the bound $\diam(G)\leq 2$ is sharp.
\end{theorem}

\begin{proof}
If $G$ has diameter 1, then $G$ is the complete graph $K_n$, which satisfies $\partial G=V$. Otherwise, by \cref{containment} and \cref{diameter2}, every vertex in G with eccentricity greater than 1 is a Steinerberger boundary vertex. If $G$ has no vertices of eccentricity 1, then $\partial G=V$. If $G$ has two or more vertices of eccentricity 1, then by \cref{diameter2}, these vertices are in $\partial G$. Thus, $\partial G=V$. 

Now suppose that $G$ has a single vertex $v$ of eccentricity $1$. Then $N_G(v)=V\setminus \{v\}$, implying $d(v,w)=1$ for all $w\in V\setminus \{v\}$.
Let $u\in N_G(v)$ satisfy $\beta(v)=\beta(v,u)$. Because no vertex other than $v$ has eccentricity $1$, we know that $\ecc(u)\geq 2$. In particular, there exists at least one vertex $x\in N(v)$ such that $d(u,x)= 2$. Therefore,
\begin{align*}
        \beta_{G}(v) & =\beta_{G}(v,u)
        =\sum_{w\in N_G(v)}[d(v,u)-d(w,u) ]\\
        & = [d(v,u)-d(u,u)]+[d(v,u)-d(x,u)]  +\sum_{w\in N_G(v)\setminus \{u,x\}}[d(v,u)-d(w,u)] \\
        &= 1 + (-1) +\sum_{w\in N_G(v)\setminus \{u,x\}}[1-d(w,u)]
        \leq 0. 
\end{align*}

To show that our bound is sharp, let $G=(V,E)$ be the $n$-barbell graph for $n\geq 2$, which is formed by adding an edge between two disjoint copies of $K_n$, as shown in \cref{barbell}. Observe that $\diam(G)=3$. Denote the two $K_n$ graphs $G_1=(V_1,E_1),G_2=(V_2,E_2)$, and let $v_1\in G_1$, $v_2\in G_2$ be the vertices where an edge is added. Since $\beta_{G_1}(v_1)=\beta_{G_2}(v_2)=1$, \cref{edgeattachment}, implies that $\partial G = (\partial G_1\cup\partial G_2)\setminus \{v_1,v_2\}$. Hence, $|\partial G|=|V|-2$.
\end{proof}

\begin{figure}[h!]
    \centering
    \scalebox{0.5}{
    \begin{tikzpicture}
    \node[shape=circle,draw=black, fill=red] (A) at (0,1) {};
    \node[shape=circle,draw=black, fill=red] (B) at (-1,0) {};
    \node[shape=circle,draw=black, fill=red] (C) at (0,-1) {};
    \node[shape=circle,draw=black, fill=blue] (D) at (1,0) {};
    \node[shape=circle,draw=black, fill=red] (A2) at (4,1) {};
    \node[shape=circle,draw=black, fill=blue] (B2) at (3,0) {};
    \node[shape=circle,draw=black, fill=red] (C2) at (4,-1) {};
    \node[shape=circle,draw=black, fill=red] (D2) at (5,0) {};
    \path (A) edge node {} (B);
    \path (A) edge node {} (C);
    \path (A) edge node {} (D);
    \path (B) edge node {} (C);
    \path (B) edge node {} (D);
    \path (C) edge node {} (D);
    \path (D) edge node {} (B2);
    \path (A2) edge node {} (B2);
    \path (A2) edge node {} (C2);
    \path (A2) edge node {} (D2);
    \path (B2) edge node {} (C2);
    \path (B2) edge node {} (D2);
    \path (C2) edge node {} (D2);
\end{tikzpicture}}
    \caption{The 4-barbell graph with $\partial G$ shown in red.}
    \label{barbell}
\end{figure}
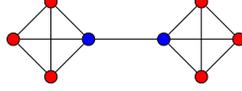

\begin{corollary}
Let $G=(V,E)$ be a graph on at least two vertices with minimum degree $\delta(G)\geq \frac{1}{2}(|V|-1)$. If $G$ has a single vertex $u$ of eccentricity $1$, then $\partial G=V\setminus \{u\}$. Otherwise, $\partial G=V$. Furthermore, the bound $\frac{1}{2}(|V|-1)$ is sharp.
\end{corollary}

\begin{proof}
Suppose $v\in G$ has $\ecc(v)>1$. We claim that $\ecc(v)=2$ by showing that $d(v,u)\leq 2$ for any $u\in V$. Let $u\in V$ be any vertex. If $u\in N_G(v)$, then $d(v,u)=1$. Otherwise, $u\notin N_G(v)$ and $v\notin N_G(u)$. The assumption $\delta(G)\geq \frac{1}{2}(|V|-1)$ implies $|N_G(v)|\geq \frac{1}{2}(|V|-1)$ and $|N_G(u)|\geq \frac{1}{2}(|V|-1)$. If $N_G(v)\cap N_G(u)=\emptyset$, then
\[|V|\geq |N_G(u)\cup N_G(v)\cup \{u,v\}|=|N_G(u)|+|N_G(v)|+2 \geq |V|+1,\]
which is a contradiction. Hence, $d(v,u)\leq 2$ for all $u$ and $\ecc(v)=2$. We conclude that when $\diam(G)\neq 1$, it must be that $\diam(G)=2$. The result now follows from \cref{diameter}, where the $n$-barbell graph also establishes that the bound $\frac{1}{2}(|V|-1|)$ is sharp, as the minimum degree in the $n$-barbell graph is $\frac{1}{2}|V|-1<\frac{1}{2}(|V|-1)$.
\end{proof}

\section{Open questions}\label{futurework}

We have characterized the graphs with $|\partial G|$ at most 4. Hence, we propose the natural next step.

\begin{problem}
Classify connected graphs $G=(V,E)$ for which $|\partial G|=5$.
\end{problem}

Our characterization of $|\partial G|=4$ relied on a characterization of $|(\partial G)'|=4$. However, a characterization of graphs with $|(\partial G)'|=5$ is not currently known. Hence, this problem requires new methods for studying Steinerberger boundary vertices. Since the Steinberberger boundary has more vertices, we expect the characterization problem to be easier than for the CEJZ boundary. 

In \cref{largeboundary}, we described some graphs where all or almost all vertices are in $\partial G$. We propose identifying some additional cases when this occurs. Note that a complete characterization is likely difficult. Data on random graphs suggests that they often consist entirely of boundary vertices. 
\begin{problem}
Describe some additional cases when $|\partial G|=|V|-1$ or $|\partial G|=|V|$.
\end{problem}

In \cref{stability}, we characterized when vertices of degree 2 are in $\partial G$, and one case of our results in \cref{largeboundary} was on graphs with large minimum degree. We propose a problem related to vertices with low degree.

\begin{problem}
Characterize when vertices of degree $3$ or $4$ are in $\partial G$. Apply this characterization to find properties of $\partial G$ for graphs with maximum degree $\Delta =3$ or $\Delta= 4$.
\end{problem}

Finally, we propose a question of the boundary stability number, a central tool in establishing our results. To prove our characterization of graphs with small boundary, we explicitly characterized the effect on boundary stability number when adding an edge between two graphs, and we described a case when boundary vertices of a subgraph are also boundary vertices in the graph itself. One natural question is the effect on boundary stability number for other operations.

\begin{problem}
Describe the effect on the boundary stability number of a vertex under other graph operations, such as edge contraction and Cartesian products.
\end{problem}

\section*{Acknowledgements}
We would like to thank Stefan Steinerberger for suggesting this problem, helpful discussions, and valuable feedback. We would also like to thank Catherine Babecki for valuable feedback and Sam Millard for helpful discussions. Finally, we would like to thank the Washington eXperimental Mathematics Lab for organizing and supporting this project.

\printbibliography

\end{document}